%% file: singular_limit_lane_changing.tex
\documentclass[3p]{elsarticle}
\usepackage[utf8]{inputenc}
\usepackage{amsmath,amssymb,amsfonts}

\usepackage{amsthm}
\usepackage{mathtools}
\usepackage{dsfont}
\usepackage{enumitem}
\usepackage{ulem}
\usepackage{hyperref}
\usepackage[capitalize,nameinlink]{cleveref}
\usepackage{todonotes}
\usepackage{mathrsfs}
\usepackage{amsmath}

\input{definitions}

\DeclareMathOperator{\supp}{supp}
\newtheorem{assumption}{Assumption}
\newtheorem{definition}{Definition}
\newtheorem{remark}{Remark}
\newtheorem{theorem}{Theorem}[section]
\newtheorem{corollary}{Corollary}[theorem]
\newtheorem{lemma}[theorem]{Lemma}
\newtheorem*{problem1}{Nonlocal problem}
\newtheorem*{problem2}{Local problem}

\title{On the singular limit problem in nonlocal balance laws: Applications to nonlocal lane-changing traffic flow models}
\author[<1>]{Felisia Angela Chiarello}

\ead{<felisiaangela.chiarello@univaq.it>}

\affiliation[<1>]{organization={University of L'Aquila, Department of Engineering and Information Science and Mathematics,
(DISIM)},
            addressline={Via Vetoio, Ed. Coppito 1}, 
            city={L'Aquila},
            postcode={67100}, 
            country={Italy}}

\author[<2>]{Alexander Keimer}

\affiliation[<2>]{organization={Friedrich-Alexander-Universität Erlangen-Nürnberg (FAU), Department of Mathematics},
            addressline={Cauerstr. 11}, 
            city={Erlangen},postcode={91058}, 
            country={Germany}}
 \ead{<alexander.keimer@fau.de>}

\date{\today}
\allowdisplaybreaks
\begin{document}
\begin{abstract}
 We present a convergence result from nonlocal to local behavior for a system of nonlocal balance laws. The velocity field of the underlying conservation laws is diagonal. In contrast, the coupling to the remaining balance laws involves a nonlinear right-hand side that depends on the solution, nonlocal term, and other factors. The nonlocal operator integrates the density around a specific spatial point, which introduces nonlocality into the problem.
 Inspired by multi-lane traffic flow modeling and lane-changing, the nonlocal kernel is discontinuous and only looks downstream.
 In this paper, we prove the convergence of the system to the local entropy solutions when the nonlocal operator (chosen to be of an exponential type for simplicity) converges to a Dirac distribution.
Numerical illustrations that support the main results are also presented.
\end{abstract}
\maketitle
\section{Introduction and problem setup}
Conservation laws with nonlocal fluxes are frequently used in vehicular traffic modeling. These models aim to describe drivers who adjust their velocity based on conditions ahead of them,see~\cite{Chiarello2023KRM, chiarello, chiarello2019non-local, chiarelloFriedrichGoatinGK,  friedrich2018godunov, pflug,keimer1}. 
There are general existence and uniqueness results for nonlocal conservation laws, as discussed in \cite{teixeira, pflug} for scalar equations in one space dimension, \cite{colombo, spinola} for multi-dimensional scalar equations, and \cite{aggarwal} for multi-dimensional systems. Two different primary approaches are commonly employed to establish solutions for these models: One approach provides suitable compactness estimates for a sequence of approximate solutions constructed through finite volume schemes, as in \cite{ blandin2016well, friedrich2018godunov, chiarello}. The other approach relies on characteristics and fixed-point theorems, as proposed in \cite{pflug, spinola}. 
Nonlocal conservation laws on a bounded domain have been studied in \cite{Filippis, goatin2019wellposedness, keimer1}, and in \cite{colombo2018nonlocal} for multi-dimensional nonlocal systems using similar methods as described above.
This study focuses on the singular limit problem of nonlocal conservation laws within the context of systems consisting of two (or more) equations. Specifically, we aim to establish the convergence of nonlocal solutions to the entropy-admissible solution of the local conservation law. This convergence occurs when we replace the convolution kernel with a Dirac delta function.  
This problem was initially posed in~\cite{amorim}, where the authors conducted a numerical investigation.
Subsequently, several authors studied the nonlocal-to-local convergence for the general scalar one-dimensional case without specific assumptions regarding the kernel function and the initial density. In particular, some counter-examples rule out convergence in the general case, see~\cite{ColomboCrippaMarconiSpinolo2021}. 
On the contrary, within the specific framework of traffic models, which includes anisotropic convolution kernels and nonnegative density, the singular limit has been established in the scalar case for nonlocal conservation laws. This has been achieved in the case of the exponential kernel~\cite{coclite2020general} or by imposing monotonicity requirements on the initial datum~\cite{pflug4}. Recently, a more general result was obtained in~\cite{Marconi2023},
which considers the convexity assumption for the convolution kernels. 
In~\cite{bressan}, the authors demonstrated nonlocal-to-local convergence by considering an initial datum with bounded total variation bounded away from zero and an exponential weight. Moreover, the group established that the solution approaches an entropic state in the limit, assuming $V$ is an affine function. This extension of the result in~\cite{bressan2021entropy} applies to more general fluxes.
In \cite{keimer42}, the authors studied the same singular limit problem but for kernels with fixed support. They obtained the convergence to the local entropy solution in these cases.

In addition, there is also a recent study on the nonlocal \(p-\)norm \cite{amadori2023nonlocal}, where, under rather general assumptions and for sufficiently large \(p\) large an Oleinik \cite{oleinik_english,oleinik} type inequality is derived. This inequality ensures the immediate convergence to the local entropy solution.
Such an Oleinik estimate had also been obtained for the earlier mentioned ``classical'' singular limit problem in \cite{coclite2023oleinik} using additional constraints on the involved velocities and/or the initial datum.

However, none of the previously mentioned studies have addressed systems of nonlocal balance laws and their singular limit, which is one of the reasons why we explore these in this paper.
We obtain a convergence result with potential applications in traffic models if we consider a \textit{system} of nonlocal balance laws (two equations) with lane-changing functions on the right-hand side and exponential kernels in the flux functions. This can be formulated as follows:
\begin{align}
    \partial_{t}\brho + \partial_{x}\big(\boldsymbol{V}(\gamma\ast\brho)\brho\big)=\boldsymbol{S}(\brho,\gamma\ast\brho)\quad \overset{\gamma\rightarrow\delta}{\longrightarrow} \quad \partial_{t}\brho + \partial_{x}\big(\boldsymbol{V}(\brho)\brho\big)=\boldsymbol{S}(\brho,\brho)\label{eq:problem}
\end{align}
with the density \(\brho:\OT\rightarrow \R^{2}\), \(\gamma\) signifying an exponential one-sided kernel, and \(\boldsymbol{V}:\R^{2}\rightarrow\R^{2}\) a ``diagonal'' velocity function \(\boldsymbol{S}:\R^{4}\rightarrow\R^{2}\) a ``semi-linear'' right-hand side (for the precise statement see \cref{ass:general} and \cref{eq:nonlocal_system}, \cref{eq:local_system}).
To our knowledge, this represents the first instance of a nonlocal-to-local convergence result for such systems.
Coupling between the equations of the system appears \textit{only} on the right-hand side, which means that some of the well-known methods for transitioning to the local limit remain applicable. As an application, we consider a traffic flow model with two lanes and lane-changing functions. However, our analysis is not limited to a system of two equations; we maintain the two-equation system solely for simplicity.
The approach taken in this paper is as follows: we obtain a uniform Total Variation (\(TV\)) bound of the nonlocal terms as well as a maximum principle. These findings enable us to transition to the limit in the weak formulation. Furthermore, we can demonstrate the entropy admissibility, akin to the scalar case presented in \cite{bressan2021entropy}. 

The paper is organized as follows:  \cref{sec:modelling} presents the model in the nonlocal and local settings. In \cref{sec:well_posedness}, we revisit some well-posedness results, while in \cref{sec:singularl_limit}, we demonstrate how to transition to the limit for $\eta \to 0$. This is accomplished by recovering uniform bounds on the total variation of the nonlocal operators and introducing a compactness argument. \Cref{sec:numerics} is dedicated to numerical simulations that support the analytical results. Lastly, \cref{sec:conclusions} concludes the paper by outlining some remaining problems.

\section{Modeling and fundamental assumptions}\label{sec:modelling}
As mentioned above, our analysis will be limited to two nonlocal scalar balance laws coupled via the right-hand side. This results in a system of nonlocal balance laws that can model lane-changing with macroscopic traffic flow equations.

In this context, it may be helpful to be aware of some classical assumptions related to the involved velocity functions, initial data, etc. We refer the reader to \cref{eq:nonlocal_system} and \cref{eq:local_system}, where the introduced functions were used.

\begin{assumption}[General assumptions regarding the utilized data]\label{ass:general}
The following was assumed:
\begin{description}
    \item[Lane-wise velocities:] \(V_{1},V_{2}\in W^{2,\infty}(\R):\ V_{1}'\leqq 0\geqq V_{2}'\)
    \item[Maximum lane densities:] \(\exists \brho_{\max}\in\R^{2}_{>0}\)
    \item[Initial datum:] \(\brho_{0}\in L^{\infty}\Big(\R;\big[0,\brho_{\max}^{1}\big]\times\big[0,\brho_{\max}^{2}\big]\Big)\cap TV\big(\R;\R^{2}\big)\)
    \item[Nonlocal impact:] \(\eta\in\R_{>0}\)
    \item[RHS, lane changing:]
    \[
S\big(\brho, \cW_\eta[\brho],x\big)= \Big(\tfrac{\brho^{2}}{\brho_{\max}^{2}} - \tfrac{\brho^{1}}{\brho_{\max}^{1}}\Big) H(\cW_\eta[\brho],x),\qquad x\in\R
\]
with \(H\in W^{1,\infty}_{\loc}(\R^{3};\R_{\geq 0})\) such that \(\exists (\mathcal{H},\mathcal{H}_{1},\mathcal{H}_{2}, \mathcal{H}_{BV})\in\R^{4}_{\geq0}:\)
\begin{align*}
\ \|H\|_{L^{\infty}((0,|\brho_{\max}|_{\infty})\times(0,|\brho_{\max}|_{\infty})\times\R)}\leq \mathcal{H}\ &\wedge\ \|\partial_{1}H\|_{L^{\infty}((0,|\brho_{\max}|_{\infty})\times(0,|\brho_{\max}|_{\infty})\times\R)}\leq \mathcal{H}_{1}\\
\wedge\ \|\partial_{2}H\|_{L^{\infty}((0,|\brho_{\max}|_{\infty})\times(0,|\brho_{\max}|_{\infty})\times\R)}\leq \mathcal{H}_{2}\ &\wedge\ 
\|H\|_{L^{\infty}((0,|\brho_{\max}|_{\infty})\times(0,|\brho_{\max}|_{\infty});BV(\R))}\leq \mathcal{H}_{BV(\R)}.
\end{align*}
\end{description}
Thereby, we define \(TV(\R)\coloneqq\{f \in L^1_{loc}(\R): |f|_{TV(\R)}<\infty\}\) and \(BV(\R)\coloneqq\{f \in L^1(\R): |f|_{TV(\R)}<\infty\}\) and the considered space-time horizon \(\OT\coloneqq (0,T)\times\R\) for \(T\in\R_{>0}\).
\end{assumption}
\begin{remark}[Reasonableness of \cref{ass:general}]
The assumption of the velocities being monotonically decreasing is reasonable in traffic flow modeling and one of the main reasons why a maximum principle can hold (see \cref{theo:nonlocal_existence_uniqueness_maximum_principle}).
The canonical assumption that the initial data set is essentially bounded and nonnegative is established. However, one might question the necessity of assuming\(TV\) regularity in addition to these criteria. As we later aim for uniform \(TV\) bounds in the nonlocal term, this assumption is required because particular nonlocal equations do not possess the well-known \(BV\) regularization (for strictly convex/concave flux).
The nonlocal impact represents how far downstream traffic affects the velocity. Because we use an exponential kernel (see \cref{defi:nonlocal_system}), the look-ahead is always infinite, but for \(\eta\) small, it is small and tends to be more localized.
Finally, the R.H.S.\ models the potential lane change from one lane to another. It already encodes the requirement that if one road is empty, density can only come from the other road. In addition, it allows the lane change to be dependent on the location.
In addition, the term \(H\) represents how the density exchange between lanes scales with regard to the density ahead. This can also be interpreted as velocity scaling.
However, this condition can be considered restrictive as it disallows lane-changing on \(\R\) and only permits it in a way that 
\begin{equation}
\|H\|_{L^{\infty}((0,|\brho_{\max}|_{\infty})\times(0,|\brho_{\max}|_{\infty});BV(\R))}\leq \mathcal{H}_{BV(\R)}\label{eq:bullshit}
\end{equation}
holds.
This condition could be removed if we would either assume that the nonlocal kernel \(\gamma\) in \cref{eq:problem} is compactly supported (and not of exponential type like in this contribution (compare \cref{eq:nonlocal_system})) or that the initial datum is in \(L^{1}(\R)\) and not -- as currently assumed -- ``only'' in \(L^{\infty}(\R)\cap TV(\R)\). In both cases, both the total variation estimates in \cref{thm:tv_bound} and the compactness in \cref{thm:compactness} could then be established as well, and \cref{eq:bullshit} would not be required.

In conclusion, one can state that none of the assumptions are restrictive for applications in traffic flow modeling.
\end{remark}

The system of nonlocal balance laws considered in this manuscript can be expressed as follows:
\begin{problem1}[The nonlocal system of balance laws]\label{defi:nonlocal_system}
    Let \cref{ass:general} hold, and consider the ``weakly'' coupled (via the right-hand side) system 
\begin{equation}
\begin{aligned}
    \partial_{t}\brho^{1}(t,x)+\partial_{x}\Big(V_{1}(\cW_\eta[\brho^{1}](t,x))\brho^{1}(t,x)\Big)&=S\big(\brho(t,x),\cW_\eta[\brho](t,x),x\big),&& (t,x)\in(0,T)\times\R\\
    \partial_{t}\brho^{2}(t,x)+\partial_{x}\Big(V_{2}(\cW_\eta[\brho^{2}](t,x))\brho^{2}(t,x)\Big)&=-S\big(\brho(t,x),\cW_\eta[\brho](t,x),x\big),&& (t,x)\in (0,T)\times\R\\
    \brho(0,x)&=\brho_{0}(x),&& x\in\R\\
    \cW_\eta[\rho](t,x)&=\tfrac{1}{\eta}\int_{x}^{\infty}\exp\left(\tfrac{x-y}{\eta}\right)\rho(t,y)\dd y,&& (t,x)\in (0,T)\times\R.
    \end{aligned}
    \label{eq:nonlocal_system}
\end{equation}
Then, we call \(\cW_\eta\) the nonlocal operator , defined for \(\rho\in C\big([0,T];L^{1}_{\loc}(\R)\big)\cap L^{\infty}((0,T);L^{\infty}(\R)),\) and \(\cW_\eta[\brho](t,x)=\big(\cW_\eta[\brho^1],\cW_\eta[\brho^2]\big)(t,x),\ (t,x)\in (0,T)\times\R\) the vector of nonlocal impact. \(\brho=(\brho^1,\brho^2)\) is named vector of solutions of
the \textbf{system of nonlocal balance laws} modeling lane-changing with two lanes.
\end{problem1}

Because we are investigating the singular limit problem for \cref{defi:nonlocal_system}, it becomes necessary to define the corresponding local system. We detail this in the following sections:
\begin{problem2}[The corresponding local system of balance laws]\label{defi:local_system}
Let \cref{ass:general} hold, and we call the ``weakly'' coupled (via the right-hand side) system
\begin{equation}
\begin{aligned}
        \partial_{t}\brho^{1}(t,x)+\partial_{x}\Big(V_{1}(\brho^{1}(t,x))\brho^{1}(t,x)\Big)&=S\big(\brho(t,x),\brho(t,x),x\big),&&(t,x)\in (0,T)\times\R\\
        \partial_{t}\brho^{2}(t,x)+\partial_{x}\Big(V_{2}(\brho^{2}(t,x))\brho^{2}(t,x)\Big)&=-S\big(\brho(t,x),\brho(t,x),x\big),&& (t,x)\in (0,T)\times\R\\
        \brho(0,x)&=\brho_{0}(x),&& x\in\R
        \end{aligned}
    \label{eq:local_system}
    \end{equation}
    the system of local balance laws, which models lane-changing for two lanes.
    \end{problem2}
    Having laid out the underlying assumptions and the dynamics under consideration, we now turn our attention to the well-posedness, i.e., \ the existence and uniqueness of solutions.
\section{Well-posedness of the system of (non)local conservation laws}\label{sec:well_posedness}
To ensure the well-posedness of the local equations, i.e., \ the existence and uniqueness of solutions, we need to first define an Entropy condition. This condition helps identify the physically meaningful solutions among the potentially infinite many weak solutions. Because the system is only weakly coupled via the right-hand side, we can employ scalar entropy conditions similar to those used in~\cite{holden2019models}.
\begin{definition}[Entropy conditions for local conservation laws]\label{defi:entropy}
Let \cref{defi:local_system} be defined for
        \[
 \alpha\in C^{2}(\R) \text{ convex, } \beta_i'\equiv\alpha'\cdot f_i' \text{ where } f_i\equiv (\cdot)V_i(\cdot),\text{ on } \R,\ i\in\{1,2\},\ \text{for } \phi\in C^{1}_{\text{c}}((-42,T)\times\R;\R_{\geq0})
\]
 and for \(\brho^{1},\brho^{2}\in C\big([0,T];L^{1}_{\loc}(\R)\big)\)
\begin{align*}
\mEF_1[\phi,\alpha,\brho^1]&\:\iint_{\OT}\alpha\big(\brho^1(t,x)\big)\phi_{t}(t,x)+\beta_1\big(\brho^1(t,x)\big)\phi_{x}(t,x)\dd x\dd t+\int_{\R}\alpha\big(\brho^1_{0}(x)\big)\phi(0,x)\dd x\\
&\quad-\int_{\OT}\alpha'\big(\brho^1(t,x)\big) S\big(\brho^1(t,x),\brho^2(t,x),\brho^1(t,x),\brho^2(t,x), x\big)\phi(t,x)\dd x\dd t\\
\mEF_2[\phi,\alpha,\brho^2]&\:\iint_{\OT}\alpha\big(\brho^2(t,x)\big)\phi_{t}(t,x)+\beta_2\big(\brho^2(t,x)\big)\phi_{x}(t,x)\dd x\dd t+\int_{\R}\alpha\big(\brho^2_{0}(x)\big)\phi(0,x)\dd x\\
&\quad+\int_{\OT}\alpha'\big(\brho^2(t,x)\big) S\big(\brho^1(t,x),\brho^2(t,x),\brho^1(t,x),\brho^2(t,x), x\big)\phi(t,x)\dd x\dd t.
\end{align*}
Then, \(\brho_{*}\in C\big([0,T];L^{1}_{\loc}(\R;\R^{2})\big)\) is called an entropy solution if it satisfies for \(i\in\{1,2\}\)
\[
\mEF_i[\phi,\alpha,\brho^i_{*}]\geq0\quad  \forall \phi\in C^{1}_{\text{c}}\big((-42,T)\times\R;\R_{\geq0}\big)\ \forall \alpha\in C^{2}(\R) \text{ convex, with }  \beta_i'\equiv\alpha'\cdot f_i'.
\]
\end{definition}
After identifying the appropriate entropy condition, we can establish the existence and uniqueness of the local system as explained in the following:
\begin{theorem}[Existence, Uniqueness \&\ Maximum principle of the local system]
Let \cref{ass:general} hold. Then, there exists a unique, weak, entropy solution \(\brho_{*}\in C\big([0,T];L^{1}_{\loc}(\R;\R^{2})\big)\cap L^{\infty}\big((0,T);L^{\infty}(\R;\R^{2})\big)\) in the sense of \cref{defi:entropy} to \cref{defi:local_system}, such that 
\[
0\leq \brho^i(t,x)\leq  \brho^i_{\max},\ i\in\{1,2\},\ (t,x)\in\OT \text{ a.e.}
\]
with \(\brho^{i}_{\max}\) as in \cref{ass:general}.
\end{theorem}
\begin{proof}
The existence and uniqueness of solutions to the local system~\eqref{defi:local_system} can be established using the results presented in~\cite{Rohde1998}. This research examined a class of weakly coupled hyperbolic multi-dimensional systems characterized by source terms dependent on unknowns, as well as spatial and temporal variables. Note that \cite[Assumption 1.1]{Rohde1998} is quite stringent, but the assumption can be relaxed according to the same author. 
For further reference, see the proof presented in~\cite{HanouzetNatalini96, HoldenKarlsenRisebro2003}, where the source term does not depend on the spatial variable. The Maximum principle, which is satisfied in this context, is derived from the parabolic approximation of the hyperbolic system as presented in~\cite{Rohde1998}.
\end{proof}
Next, we define "weak solutions" for the considered class of nonlocal conservation laws in \cref{defi:nonlocal_system}. Because the class of nonlocal conservation laws yields unique, weak solutions, there is no need to define an entropy (which is typically done in local conservation laws and particularly in \cref{defi:entropy}).
\begin{definition}[Weak solution for the system of nonlocal conservation laws]
For a system of nonlocal conservation laws as in \cref{eq:nonlocal_system} we call \((\brho^{1},\brho^{2})\in C\big([0,T];L^{1}_{\text{loc}}(\R;\R^{2})\big)\cap L^{\infty}\big((0,T);L^{\infty}(\R;\R^{2})\big)\) a weak solution to \cref{defi:nonlocal_system}, if for all \(\phi\in C^{1}_{\text{c}}\big((-42,T)\times\R\big)\) and for \(i\in\{1,2\}\) the following holds:
\begin{gather*}
    \iint_{\OT}\brho^{i}(t,x)\big(\phi_{t}(t,x)+V_{i}(\cW_\eta[\brho^{i}](t,x))\phi_{x}(t,x)\big)\dd x\dd t+\int_{\R}\phi(0,x)\brho^{i}_{0}(x)\dd x\\
    =(-1)^{i}\iint_{\OT} \phi(t,x)S\big(\brho(t,x),\cW_\eta[\brho](t,x), x\big)\dd x\dd t
\end{gather*}
and it is complemented by the nonlocal operator:
\[
\cW_\eta[\brho^{i}](t,x)\coloneqq \tfrac{1}{\eta}\int_{x}^{\infty}\exp\big(\tfrac{x-y}{\eta}\big)\brho^{i}(t,y)\dd y,\ (t,x)\in\OT,\ i\in\{1,2\}.
\]

\end{definition}
In the next theorem, we will establish the existence and uniqueness of solutions for the nonlocal balance law, as in \cref{defi:nonlocal_system}:
\begin{theorem}[Existence\ \&\ Uniqueness \&\ Maximum principle]\label{theo:nonlocal_existence_uniqueness_maximum_principle}
Let \cref{ass:general} be true. Then, there exists a unique weak \(\brho\in C\big([0,T];L^{1}_{\loc}(\R;\R^{2})\big)\cap L^{\infty}\big((0,T);L^{\infty}(\R;\R^{2})\cap TV(\R;\R^{2})\big)\) of \cref{eq:nonlocal_system} and the solution satisfies
\begin{align*}
    0\leq \brho^{i}(t,x)\leq \brho^{i}_{\max}\ (t,x)\in\OT \text{ a.e.},\ i\in\{1,2\}.
\end{align*}
   \end{theorem}
\begin{proof}
    This is a consequence of \cite[Theorem 2.15]{KeimerMultilane2022} for a small time horizon, and, thanks to the maximum principle in \cite[Theorem 3.3\ \& Lemma 3.4]{KeimerMultilane2022}, it can be extended to any finite time horizon.
\end{proof}

Another important result in this work is the stability of solutions in \(L^{1}\) and that we can approximate solutions using sufficiently smooth solutions.
\begin{lemma}[Continuous dependence of nonlocal solutions to the initial datum and smooth solutions]\label{lemma:stability}
Let the assumptions of \cref{theo:nonlocal_existence_uniqueness_maximum_principle}  be given, and assume that for \(\eps\in\R_{>0}\) the functions \(\phi_{\eps}^{1}\in C^{\infty}_{\text{c}}(\R;\R_{\geq0})\) and \(\phi_{\eps}^{2}\in C^{\infty}_{\text{c}}(\R^{2};\R_{\geq0})\) denote the standard mollifier in the sense of \cite[Remark C.18]{leoni}. We define
\[
\brho_{0,\eps}\equiv \phi_{\eps}^{1}\ast \brho_{0},\ H_{\eps}=\phi_{\eps}^{2}\ast H
\]
and call \(\brho_{\eps}\in C([0,T];L^{1}_{\loc}(\R;\R^{2}))\cap L^{\infty}((0,T);TV(\R;\R^{2}))\) the solution to the corresponding nonlocal conservation law with the initial datum \(\brho_{\eps}\) and lane-changing function \(H_{\eps}\). Then, \(\brho_{\eps}\in W^{1,\infty}_{\text{loc}}(\OT)\) and we obtain
\[
\lim_{\eps\rightarrow 0} \|\brho_{\eps}-\brho\|_{C([0,T];L^{1}(\R;\R^{2}))}=0.
\]
In particular, \(\brho_{\eps}\) is a strong solution of \cref{defi:nonlocal_system} and the nonlocal operator admits additional regularity, i.e.\ 
\[
W_{\eta}[\brho_{\eta}]\in W^{2,\infty}_{\text{loc}}(\OT;\R^{2}).
\]
\end{lemma}
\begin{proof}
The proof mainly shows that the nonlocal operator renders the velocity field of the conservation laws Lipschitz-continuous. Subsequently, one can apply classical approximation results for linear conservation laws with regard to the velocity field as well as some Gronwall estimates. We refer the reader to 
\cite{KeimerMultilane2022} and to
\cite{keimer2021discontinuous}.
\end{proof}

We also require a technical lemma, which we detail in the following:
\begin{lemma}[{\(\partial_{2}\cW_\eta[\brho^{i}]\)} vanishing at \(\infty\)]\label{lem:W_vanishing_infty}
    It holds for \(i\in\{1,2\}\) that the spatial derivative of the nonlocal term, as in \cref{eq:nonlocal_system}, vanishes at \(\infty\), i.e., \(\forall\eta\in\R_{>0},\ i\in\{1,2\}\) 
    \[
\lim_{x\rightarrow\infty} \partial_{x}\cW_\eta[\brho^{i}](t,x)=0\qquad \forall t\in[0,T].
    \]
\end{lemma}
\begin{proof}
Thanks to \cref{lemma:stability} we can assume that the nonlocal solution's initial datum is smooth, with a smoothing parameter \(\eps\in\R_{>0}\) so that the corresponding solution for \(i\in\{1,2\}\) \(\brho^{i}_{\eps}\in W^{1,\infty}(\OT)\) represents a robust solution.
Next, we can compute the derivative of the nonlocal operator and have for \((t,x)\in\OT\)
\begin{align*}
    \big|\partial_{x}\cW_\eta[\brho^{i}_{\eps}](t,x)\big|&=\tfrac{1}{\eta}\big|\cW_\eta[\brho^{i}_{\eps}](t,x)-\brho^{i}_{\eps}(t,x)\big|=\tfrac{1}{\eta}\bigg|\int_{x}^{\infty}\e^{\frac{x-y}{\eta}}\partial_{y}\brho^{i}_{\eps}(t,y)\dd y\bigg|\\
    &\leq \tfrac{1}{\eta}\int_{x}^{\infty}\e^{\frac{x-y}{\eta}}\big|\partial_{y}\brho^{i}_{\eps}(t,y)\big|\dd y\leq \tfrac{1}{\eta}\int_{x}^{\infty}|\partial_{y}\brho^{i}_{\eps}(t,y)|\dd y=\tfrac{1}{\eta}|\brho^{i}_{\eps}(t,\cdot)|_{TV(x,\infty)}.
\end{align*}
For \(x\rightarrow\infty\), the right-hand side vanishes, and thus, we obtain our claim for every \(\eps\in\R_{>0}\) as well as for the non-smoothed solution.
\end{proof}
Equipped with the well-posedness and approximation results, we can now turn to tackle the singular limit problem.
\section{The singular limit problem or nonlocal approximation of local lane-change traffic models}\label{sec:singularl_limit}
In this section, we first establish an equation solely in the nonlocal operator (similar to the approach in \cite{coclite2020general}), see \cref{lem:nonlocal_transport_equation}. This will allow us, to prove a total variation bound uniform in \(\eta\) using \cref{thm:tv_bound}. We then demonstrate that whenever a nonlocal balance law converges strongly in \(C(L^{1})\), it converges to the entropy solution (\cref{thm:entropy_admissibility}). \Cref{thm:compactness}, along with the uniform \(TV\) estimate, contributes to obtained "spatial compactness", which results in time compactness as well and leads to strong convergence in \(C(L^{1})\). Eventually, in \cref{theo:singular_limit_problem}, we collect the previously established results and obtain the singular limit convergence to the (local) entropy solution.
\subsection{Total Variation bounds uniform with respect to the nonlocal terms}
We start by formulating a Cauchy problem entirely in nonlocal terms. This approach has the advantage that the properties of the solutions \(\brho^{i}\) do not need to be studied anymore, only the properties of \(\cW[\brho]\), which turn out to behave better (one can obtain uniform \(TV\) estimates later in \cref{thm:tv_bound}).

\begin{lemma}[System of transport equations with nonlocal sources satisfied by the nonlocal operator]\label{lem:nonlocal_transport_equation}
The nonlocal terms \(\cW[\brho^{i}],\ i\in\{1,2\}\) of the system dynamics in 
\eqref{eq:nonlocal_system} are satisfied upon introducing the following abbreviations for \((t,x)\in\OT\) and \(i\in\{1,2\}\)
\begin{align}
\bW^i_{\eta}(t,x)&\coloneqq \cW[\brho^{i}](t,x),\\
\bW_{\eta}(t,x)&\coloneqq\big(\bW^1_{\eta}, \bW^2_{\eta}\big)(t,x),\\
\mathscr{S}\big(\bW_\eta,\eta\partial_{2}\bW_\eta,\cdot\big)&\coloneqq S\big(\bW^1_{\eta}-\eta \partial_{2}\bW^1_{\eta},\bW^2_{\eta}-\eta \partial_{2}\bW^2_{\eta},\bW^1_{\eta},\bW^2_{\eta}, \cdot\big),\label{eq:abbreviation_S}
\end{align}
 the coupled Cauchy problem:
\begin{equation}
\begin{aligned} \label{eq:W_system}
\partial_{t}\bW^1_{\eta}(t,x)&=-V_{1}(\bW^1_{\eta}(t,x))\partial_{x}\bW^1_{\eta}(t,x)-\tfrac{1}{\eta}\int_{x}^{\infty}\!\!\!\!\!\exp(\tfrac{x-y}{\eta})V_{1}'(\bW^1_{\eta}(t,y))\bW^1_{\eta}(t,y)\partial_{y}\bW^1_{\eta}(t,y)\dd y\\
        &\quad+\tfrac{1}{\eta}\int_{x}^{\infty}\!\!\!\!\!\exp(\tfrac{x-y}{\eta})\mathscr{S}\big(\bW_\eta(t,y),\eta\partial_{y}\bW_\eta(t,y),y\big)\dd y,\\
        \partial_{t}\bW^2_{\eta}(t,x)&=-V_{2}(\bW^2_{\eta}(t,x))\partial_{x}\bW^2_{\eta}(t,x)-\tfrac{1}{\eta}\int_{x}^{\infty}\!\!\!\!\!\exp(\tfrac{x-y}{\eta})V_{2}'(\bW^2_{\eta}(t,y))\bW^2_{\eta}(t,y)\partial_{y}\bW^2_{\eta}(t,y)\dd y\\
        &\quad-\tfrac{1}{\eta}\int_{x}^{\infty}\!\!\!\!\!\exp(\tfrac{x-y}{\eta})\mathscr{S}\big(\bW_\eta(t,y),\eta\partial_{y}\bW_\eta(t,y),y\big)\dd y,
\end{aligned}        
        \end{equation}
        which is supplemented by the following initial conditions:
        \begin{equation}
\big(\bW^1_\eta(0,x),\bW^2_\eta(0,x)\big)=\tfrac{1}{\eta}\left(\int_x^\infty\exp(\tfrac{x-y}{\eta})\brho^1_0(y) \dd y,\int_x^\infty\exp(\tfrac{x-y}{\eta})\brho_0^2(y) \dd y\right),\qquad x\in\R.\label{eq:lem:nonlocal_transport_equation_initial_datum}
        \end{equation}
\end{lemma}
\begin{proof} 
We take advantage of \cref{lemma:stability} and assume first that the initial datum is smooth enough to obtain strong solutions (we suppress the additional dependency on the regularization parameter). Then, we can compute the partial derivative with respect to \(x\) of \(\cW\), and we obtain for \((t,x)\in\OT\) and \(i\in\{1,2\}\)
\begin{equation}
\partial_{x}\bW^i_{\eta}(t,x)=\tfrac{1}{\eta}\big(\bW^{i}_{\eta}(t,x)-\brho^{i}(t,x)\big)\implies \brho^{i}(t,x)=\bW^{i}_{\eta}(t,x)-\eta\partial_{x}\bW^{i}_{\eta}(t,x).\label{eq:identity_nonlocal}
\end{equation}
        Then, we can compute the time derivative of \(\bW^{1}_{\eta}\) (and analogously, also \(\bW^{2}_{\eta}\)) and obtain
 \begin{align*}
                \partial_{t}\bW^{1}_{\eta}(t,x)&\overset{\eqref{eq:nonlocal_system}}{=}-\tfrac{1}{\eta}\int_{x}^{\infty}\exp(\tfrac{x-y}{\eta})\partial_{y}\Big(V_{1}(\bW^{1}_{\eta}(t,y))\brho^{1}(t,y)\Big)\dd y\\
        &\qquad+\tfrac{1}{\eta}\int_{x}^{\infty}\exp(\tfrac{x-y}{\eta})S\big(\brho^{1}(t,y),\brho^{2}(t,y),\bW^1_{\eta}(t,y), \bW^2_{\eta}(t,y),y\big)\dd y
        \intertext{and using partial integration}
        &=-\tfrac{1}{\eta^{2}}\int_{x}^{\infty}\exp(\tfrac{x-y}{\eta})V_{1}(\bW^{1}_{\eta}(t,y))\brho^{1}(t,y)\dd y+\tfrac{1}{\eta}V_{1}(\bW^{1}_{\eta}(t,x))\brho^{1}(t,x)\\
        &\qquad+\tfrac{1}{\eta}\int_{x}^{\infty}\exp(\tfrac{x-y}{\eta})S\big(\brho^{1}(t,y),\brho^{2}(t,y),\bW^{1}_{\eta}(t,y),\bW^{2}_{\eta}(t,y),y\big)\dd y
        \intertext{after inserting \cref{eq:identity_nonlocal} for \(\brho^{1}\) and \(\brho^{2},\) and using the notation in \cref{eq:abbreviation_S}we obtain} 
        &=-\tfrac{1}{\eta^{2}}\int_{x}^{\infty}\exp(\tfrac{x-y}{\eta})V_{1}(\bW^{1}_{\eta}(t,y))\bW^{1}_{\eta}(t,y)\dd y\\
        &\qquad+\tfrac{1}{\eta}\int_{x}^{\infty}\exp(\tfrac{x-y}{\eta})V_{1}(\bW^{1}_{\eta}(t,y))\partial_{y}\bW^{1}_{\eta}(t,y)\dd y\\
        &\qquad +\tfrac{1}{\eta}V_{1}(\bW^{1}_{\eta}(t,x))\bW^1_{\eta}(t,x)-V_{1}(\bW^{1}_{\eta}(t,x))\partial_{x}\bW^{1}_{\eta}](t,x)\\
        &\qquad+\tfrac{1}{\eta}\int_{x}^{\infty}\exp(\tfrac{x-y}{\eta})\mathscr{S}\big(\bW_{\eta}(t,y),\eta\partial_{y}\bW_{\eta}(t,y),y\big)\dd y
        \intertext{another integration by parts in the second term yields}
        &=-\tfrac{1}{\eta}\int_{x}^{\infty}\exp(\tfrac{x-y}{\eta})V_{1}'(\bW^{1}_{\eta}(t,y))\bW^{1}_{\eta}(t,y)\partial_{y}\bW^{1}_{\eta}(t,y)\dd y\\
        &\qquad -V_{1}(\bW^{1}_{\eta}(t,x))\partial_{x}\bW^{1}_{\eta}(t,x)+\tfrac{1}{\eta}\int_{x}^{\infty}\exp(\tfrac{x-y}{\eta})\mathscr{S}\big(\bW_{\eta}(t,y),\eta\partial_{y}\bW_{\eta}(t,y),y\big)\dd y.
        \end{align*}
        Repeating the same argument for \(\bW^{2}_{\eta}\) yields the claim for the robust solutions, i.e., in particular, for the smooth initial datum. However, thanks to \cref{lemma:stability}, this holds also for the general datum, which concludes the proof.
    \end{proof}
    \begin{remark}[Reasonableness of the nonlocal dynamics]
The system in \cref{eq:W_system} is for \(i\in\{1,2\}\) and \((t,x)\in\OT\) indeed a nonlocal approximation of
\begin{align*}
\partial_{t}\brho^{i}(t,x)&=-V_{i}(\brho^{i}(t,x))\partial_{x}\brho^{i}(t,x)-V_{i}'(\brho^{i}(t,x))\brho^{i}(t,x)\partial_{x}\brho^{i}(t,x)\\
&\quad +(-1)^{i+1}S\big(\brho^{1}(t,x),\brho^{2}(t,x),\brho^{1}(t,x),\brho^{2}(t,x),x\big)\\
&=\partial_{x}\big(V_{i}(\brho^{i}(t,x))\brho^{i}(t,x)\big) +(-1)^{i+1}S\big(\brho^{1}(t,x),\brho^{2}(t,x),\brho^{1}(t,x),\brho^{2}(t,x),x\big)
\end{align*}
which can be easily observed for \(\eta\rightarrow 0\).
    \end{remark}
    Following the same method of proof as in \cite{coclite2020general}, the formulation of the nonlocal terms in \cref{lem:nonlocal_transport_equation} makes it possible to derive total variation estimates directly, which are uniform in the nonlocal parameter \(\eta\).
    \begin{theorem}[Total variation bound uniform in \(\eta\)] \label{thm:tv_bound}
    Given \cref{ass:general}, the solution \(\bW_\eta\coloneqq\big(\bW^1_\eta,\,\bW^2_\eta\big)\) to the system in \cref{eq:W_system} with the initial datum, as in \cref{eq:lem:nonlocal_transport_equation_initial_datum}, satisfies the following total variation bound  \(\forall t\in[0,T]\)
    \begin{equation}
    \begin{aligned}
    \big|\bW_\eta(t,\cdot)\big|_{TV(\R;\R^{2})}&\leq \bigg(|\bq_{0}|_{TV(\R;\R^{2})}+4\Big(\tfrac{\|\brho_{\max}\|_{\infty}}{\brho_{\max}^{2}}+\tfrac{\|\brho_{\max}\|_{\infty}}{\brho_{\max}^{1}}+1 \Big) \mathcal{H}_{BV}\bigg)\\
    &\ \cdot\exp\bigg(2t \Big(\tfrac{\|\brho_{\max}\|_{\infty}\mathcal{H}_{1}}{\brho_{\max}^{2}}\!+\!\tfrac{\mathcal{H}}{\brho_{\max}^{1}}\!+\!\tfrac{\|\brho_{\max}\|_{\infty}\mathcal{H}_{1}}{\brho_{\max}^{1}}\!+\!2\mathcal{H}_{1}\!+\!\tfrac{\|\brho_{\max}\|_{\infty}\mathcal{H}_{1}}{\brho_{\max}^{1}}\!+\!\tfrac{\mathcal{H}}{\brho_{\max}^{2}}\!+\!\tfrac{\|\brho_{\max}\|_{\infty}\mathcal{H}_{1}}{\brho_{\max}^{2}}\Big)\bigg)
    \end{aligned}
\end{equation}
with the constants involved in the estimate as shown in \cref{ass:general}.
    \end{theorem}
    \begin{proof}
  Let us first assume that our initial datum is smooth, which is, thanks to \cref{lemma:stability}, not a restriction. Recalling the identities in \(\bW\) in \cref{lem:nonlocal_transport_equation} as well as the notation in \cref{eq:abbreviation_S},
  we compute at first the spatial derivative of $\partial_{t}\bW^1_{\eta}(t,x)$ and $\partial_{t}\bW^2_{\eta}(t,x)$ for \((t,x)\in\OT\) and arrive at
      \begin{equation}
      \begin{aligned}
          \partial_{t}\partial_{x}\bW^1_{\eta}(t,x)&=-\tfrac{1}{\eta^2}\!\!\!\int_{x}^{\infty}\!\!\!\!\!\!\exp(\tfrac{x-y}{\eta})V_{1}'(\bW^1_{\eta}(t,y))\bW^1_{\eta}(t,y)\partial_{y}\bW^1_{\eta}(t,y)\dd y+\tfrac{1}{\eta}V_{1}'(\bW^1_{\eta}(t,x))\bW^1_{\eta}(t,x)\partial_{x}\bW^1_{\eta}(t,x)\\
          &\qquad-\tfrac{1}{\eta}V_{1}'(\bW^1_{\eta}(t,x))\big(\partial_{x}\bW^1_{\eta}(t,x)\big)^{2}-\tfrac{1}{\eta}V_{1}(\bW^1_{\eta}(t,x))\partial_{x}^{2}\bW^1_{\eta}(t,x)\\
        &\qquad+\tfrac{1}{\eta^2}\int_x^\infty\exp(\tfrac{x-y}{\eta})\mathscr{S}\big(\bW(t,y),\eta\partial_{y}\bW(t,y),y\big)\dd y -\tfrac{1}{\eta}\mathscr{S}\big(\bW(t,x),\eta\partial_{x}\bW(t,x),x\big)\\
        \partial_{t}\partial_{x}\bW^2_{\eta}(t,x)&=-\tfrac{1}{\eta^2}\!\!\!\int_{x}^{\infty}\!\!\!\!\!\!\exp(\tfrac{x-y}{\eta})V_{2}'(\bW^2_{\eta}(t,y))\bW^2_{\eta}(t,y)\partial_{y}\bW^2_{\eta}(t,y)\dd y+\tfrac{1}{\eta}V_{2}'(\bW^2_{\eta}(t,x))\bW^2_{\eta}(t,x)\partial_{x}\bW^2_{\eta}(t,x)\\
          &\qquad-\tfrac{1}{\eta}V_{2}'(\bW^2_{\eta}(t,x))\big(\partial_{x}\bW^2_{\eta}(t,x)\big)^{2}-\tfrac{1}{\eta}V_{2}(\bW^2_{\eta}(t,x))\partial_{x}^{2}\bW^2_{\eta}(t,x)\\
        &\qquad-\tfrac{1}{\eta^2}\int_x^\infty\exp(\tfrac{x-y}{\eta})\mathscr{S}\big(\bW(t,y),\eta\partial_{y}\bW(t,y),y\big) \dd y+\tfrac{1}{\eta}\mathscr{S}\big(\bW(t,x),\eta\partial_{x}\bW(t,x),x\big).
      \end{aligned}
      \label{eq:partial_x_partial_t_W}
      \end{equation}
      Next, we compute the total variation of \(\cW[\brho]\), i.e.,\ \(|\bW^1_{\eta}(t,\cdot)|_{TV(\R)}+ |\bW^2_{\eta}(t,\cdot)|_{TV(\R)}\) starting with \(|\bW^1_{\eta}(t,\cdot)|_{TV(\R)}\) 
\begin{align}
\notag
    &\tfrac{\dd}{\dd t}\int_{\R}|\partial_{x}\bW^1_{\eta}(t,x)|\dd x=\int_{\R}\sgn(\partial_{x}\bW^1_{\eta}(t,x))\partial_{t}\partial_{x}\bW^1_{\eta}(t,x)\dd x\notag\\
    &\overset{\eqref{eq:partial_x_partial_t_W}}{=}-\tfrac{1}{\eta^2}\int_\R \sgn(\partial_{x}\bW^1_{\eta}(t,x)) \int_{x}^{\infty}\exp(\tfrac{x-y}{\eta})V_{1}'(\bW^1_{\eta}(t,y))\bW^1_{\eta}(t,y)\partial_{y}\bW^1_{\eta}(t,y)\dd y \dd x\\
    \notag
    &\quad +\tfrac{1}{\eta}\int_{\R}\sgn(\partial_{x}\bW^1_{\eta}(t,x))V_{1}'(\bW^1_{\eta}(t,x))\bW^1_{\eta}(t,x)\partial_{x}\bW^1_{\eta}(t,x)\dd x \\
    \notag
    &\quad -\tfrac{1}{\eta}\int_{\R}\sgn(\partial_{x}\bW^1_{\eta}(t,x)) V_{1}'(\bW^1_{\eta}(t,x))\big(\partial_{x}\bW^1_{\eta}(t,x)\big)^{2}\dd x\\
    \notag
    &\quad -\tfrac{1}{\eta}\int_{\R}\sgn(\partial_{x}\bW^1_{\eta}(t,x))V_{1}(\bW^1_{\eta}(t,x))\partial_{x}^{2}\bW^1_{\eta}(t,x)\dd x\\
    \notag
    &\quad +\tfrac{1}{\eta^2} \int_{\R}\sgn(\partial_{x}\bW^1_{\eta}(t,x)) \int_x^\infty\exp(\tfrac{x-y}{\eta})\mathscr{S}\big(\bW(t,y),\eta\partial_{y}\bW(t,y),y\big) \dd y\dd x\\
    \notag
    &\quad -\tfrac{1}{\eta}\int_{\R}\sgn(\partial_{x}\bW^1_{\eta}(t,x))\mathscr{S}\big(\bW(t,x),\eta\partial_{x}\bW(t,x),x\big)\dd x.
    \intertext{Performing an integration by parts in the fourth term, using \(\sgn(\partial_{x}\bW^1_{\eta}(t,x))\partial_{x}^{2}\bW_{1}(t,x)=\tfrac{\dd}{\dd x}|\partial_{x}\bW^1_{\eta}(t,x)|\)}
    &=-\tfrac{1}{\eta^2}\int_\R \sgn(\partial_{x}\bW^1_{\eta}(t,x)) \int_{x}^{\infty}\exp(\tfrac{x-y}{\eta})V_{1}'(\bW^1_{\eta}(t,y))\bW^1_{\eta}(t,y)\partial_{y}\bW^1_{\eta}(t,y)\dd y \dd x\\
    \notag
    &\quad +\tfrac{1}{\eta}\int_{\R}|\partial_{x}\bW^1_{\eta}(t,x)|V_{1}'(\bW^1_{\eta}(t,x))\bW^1_{\eta}(t,x)\dd x \\
    \notag 
    &\quad +\tfrac{1}{\eta^2} \int_{\R}\sgn(\partial_{x}\bW^1_{\eta}(t,x)) \int_x^\infty\exp(\tfrac{x-y}{\eta})\mathscr{S}\big(\bW(t,y),\eta\partial_{x}\bW(t,y),y\big) \dd y\dd x\\
    \notag
    &\quad -\tfrac{1}{\eta}\int_{\R}\sgn(\partial_{x}\bW^1_{\eta}(t,x))\mathscr{S}\big(\bW(t,x),\eta\partial_{x}\bW(t,x),x\big)\dd x
    \intertext{and exchanging the order of integration}
    &\leq-\tfrac{1}{\eta^2}\int_\R V_{1}'(\bW^1_{\eta}(t,y))\bW^1_{\eta}(t,y)|\partial_{y}\bW^1_{\eta}(t,y)|\int_{-\infty}^{y}\exp(\tfrac{x-y}{\eta})\dd x \dd y\\
    \notag
    &\quad +\tfrac{1}{\eta}\int_{\R}|\partial_{x}\bW^1_{\eta}(t,x)|V_{1}'(\bW^1_{\eta}(t,x))\bW^1_{\eta}(t,x)\dd x \\
    \notag
    &\quad +\tfrac{1}{\eta^2} \int_{\R}\sgn(\partial_{x}\bW^1_{\eta}(t,x)) \int_x^\infty\exp(\tfrac{x-y}{\eta})\mathscr{S}\big(\bW(t,y),\eta\partial_{y}\bW(t,y),y\big) \dd y\dd x\\
    \notag
    &\quad -\tfrac{1}{\eta}\int_{\R}\sgn(\partial_{x}\bW^1_{\eta}(t,x))\mathscr{S}\big(\bW(t,x),\eta\partial_{x}\bW(t,x),x\big)\dd x\\
    \notag
    &=\tfrac{1}{\eta^2} \int_{\R}\sgn(\partial_{x}\bW^1_{\eta}(t,x)) \int_x^\infty\exp(\tfrac{x-y}{\eta})\mathscr{S}\big(\bW(t,y),\eta\partial_{y}\bW(t,y),y\big) \dd y\dd x\\
    \notag
    &\quad -\tfrac{1}{\eta}\int_{\R}\sgn(\partial_{x}\bW^1_{\eta}(t,x))\mathscr{S}\big(\bW(t,x),\eta\partial_{x}\bW(t,x),x\big)\dd x
    \notag
    \intertext{also, an integration by parts in the first term with regard to the exponential function yields}
    &=\tfrac{1}{\eta} \int_{\R}\sgn(\partial_{x}\bW^1_{\eta}(t,x)) \int_x^\infty\exp(\tfrac{x-y}{\eta})\tfrac{\dd}{\dd y}\mathscr{S}\big(\bW(t,y),\eta\partial_{y}\bW(t,y),y\big) \dd y\dd x.\label{eq:step_in_TV_estimate}
\end{align}
We still need to investigate the spatial derivative of the source term \(\mathscr{S}\) in greater detail. Recalling its definition in \cref{eq:abbreviation_S} and \cref{ass:general},
 we can compute for \((t,y)\in\OT\) as follows:
\begin{align*}
&\tfrac{\dd}{\dd y}\mathscr{S}\big(\bW(t,y),\eta\partial_{y}\bW(t,y),y\big) \\
&=\tfrac{\dd}{\dd y}S\big(\bW^1_{\eta}(t,y)-\eta \partial_{2}\bW^1_{\eta}(t,y),\bW^2_{\eta}(t,y)-\eta \partial_{2}\bW^2_{\eta}(t,y),\bW^1_{\eta}(t,y),\bW^2_{\eta}(t,y), y\big)\\
&=\tfrac{\dd}{\dd y}\bigg(\Big(\tfrac{\bW^2_{\eta}(t,y)-\eta \partial_{2}\bW^2_{\eta}(t,y)}{\brho^{2}_{\max}}- \tfrac{\bW^1_{\eta}(t,y)-\eta \partial_{2}\bW^1_{\eta}(t,y)}{\brho^{1}_{\max}}\Big)H\big(\bW^{1}_{\eta}(t,y),\bW^{2}_{\eta}(t,y),y\big)\bigg)\\
&=\Big(\tfrac{\partial_{2}\bW^2_{\eta}(t,y)-\eta \partial_{2}^{2}\bW^2_{\eta}(t,y)}{\brho^{2}_{\max}}- \tfrac{\partial_{2}\bW^1_{\eta}(t,y)-\eta \partial_{2}^{2}\bW^1_{\eta}(t,y)}{\brho^{1}_{\max}}\Big)H\big(\bW^{1}_{\eta}(t,y),\bW^{2}_{\eta}(t,y),y\big)\\
&\quad +\Big(\tfrac{\bW^2_{\eta}(t,y)-\eta \partial_{2}\bW^2_{\eta}(t,y)}{\brho^{2}_{\max}}- \tfrac{\bW^1_{\eta}(t,y)-\eta \partial_{2}\bW^1_{\eta}(t,y)}{\brho^{1}_{\max}}\Big)\cdot \Big(\partial_1 H\big(\bW^{1}_{\eta}(t,y),\bW^{2}_{\eta}(t,y),y\big)\partial_{2}\bW^{1}_{\eta}(t,y)\\
&\qquad\qquad\qquad+\partial_{2}H\big(\bW^{1}_{\eta}(t,y),\bW^{2}_{\eta}(t,y),y\big)\partial_{2}\bW^{2}_{\eta}(t,y)+\partial_{3}H\big(\bW^{1}_{\eta}(t,y),\bW^{2}_{\eta}(t,y),y\big)\Big).
\end{align*}
Because \(\tfrac{\dd}{\dd y}\mathscr{S}\) involves higher order derivatives of \(\bW\), integration by parts is necessary, and we continue our estimate in \cref{eq:step_in_TV_estimate} by changing the order of integration to arrive at:
\begin{align*}
\notag
    \eqref{eq:step_in_TV_estimate}\leq &\tfrac{1}{\eta\brho_{\max}^{2}}\int_{\R}\partial_{2}\bW^{2}_{\eta}(t,y)H\big(\bW^{1}_{\eta}(t,y),\bW^{2}_{\eta}(t,y),y\big) \int_{-\infty}^{y}\sgn(\partial_{x}\bW^{1}_{\eta}(t,x))\exp\big(\tfrac{x-y}{\eta}\big)\dd x\dd y\\
    &\quad -\tfrac{1}{\brho_{\max}^{2}}\int_{\R}\partial_{2}^{2}\bW^{2}_{\eta}(t,y)H\big(\bW^{1}_{\eta}(t,y),\bW^{2}_{\eta}(t,y),y\big) \int_{-\infty}^{y}\sgn(\partial_{x}\bW^{1}_{\eta}(t,x))\exp\big(\tfrac{x-y}{\eta}\big)\dd x\dd y\\
    &\quad -\tfrac{1}{\eta\brho_{\max}^{1}}\int_{\R}\partial_{2}\bW^{1}_{\eta}(t,y)H\big(\bW^{1}_{\eta}(t,y),\bW^{2}_{\eta}(t,y),y\big) \int_{-\infty}^{y}\sgn(\partial_{x}\bW^{1}_{\eta}(t,x))\exp\big(\tfrac{x-y}{\eta}\big)\dd x\dd y\\
    &\quad +\tfrac{1}{\brho_{\max}^{1}}\int_{\R}\partial_{2}^{2}\bW^{1}_{\eta}(t,y)H\big(\bW^{1}_{\eta}(t,y),\bW^{2}_{\eta}(t,y),y\big) \int_{-\infty}^{y}\sgn(\partial_{x}\bW^{1}_{\eta}(t,x))\exp\big(\tfrac{x-y}{\eta}\big)\dd x\dd y\\
    &\quad +\tfrac{2}{\eta}\|\partial_{1}H\|_{L^{\infty}((0,|\brho_{\max}|_{\infty})\times(0,|\brho_{\max}|_{\infty})\times\R)}\int_{\R}\big|\partial_{2}\bW^{1}_{\eta}(t,y)\big| \int_{-\infty}^{y}\exp\big(\tfrac{x-y}{\eta}\big)\dd x\dd y\\
    &\quad +\tfrac{2}{\eta}\|\partial_{2}H\|_{L^{\infty}((0,|\brho_{\max}|_{\infty})\times(0,|\brho_{\max}|_{\infty})\times\R)}\int_{\R}\big|\partial_{2}\bW^{2}_{\eta}(t,y)\big| \int_{-\infty}^{y}\exp\big(\tfrac{x-y}{\eta}\big)\dd x\dd y\\
    &\quad +\tfrac{1}{\eta}\|H\|_{L^{\infty}((0,|\brho_{\max}|_{\infty})\times(0,|\brho_{\max}|_{\infty});TV(\R))}\sup_{y\in\R}\int_{-\infty}^{y}\exp\big(\tfrac{x-y}{\eta}\big)\dd x\dd y.
    \intertext{An integration by parts in the terms involving \(\partial_{2}^{2}\bW_{\eta}^{i},\ i\in\{1,2\}\) and subsequent straightforward computations yield}
    &\leq\|H\|_{L^{\infty}((0,|\brho_{\max}|_{\infty})\times(0,|\brho_{\max}|_{\infty})\times\R)}\tfrac{1}{\eta\brho_{\max}^{2}}\int_{\R}|\partial_{2}\bW^{2}_{\eta}(t,y)|\int_{-\infty}^{y}\exp\big(\tfrac{x-y}{\eta}\big)\dd x\dd y\\
    &\quad -\tfrac{1}{\brho_{\max}^{2}}\lim_{y\rightarrow\infty}\partial_{2}\bW^{2}_{\eta}(t,y)H\big(\bW^{1}_{\eta}(t,y),\bW^{2}_{\eta}(t,y),y\big) \int_{-\infty}^{y}\sgn(\partial_{x}\bW^{1}_{\eta}(t,x))\exp\big(\tfrac{x-y}{\eta}\big)\dd x\dd y\\
    &\quad +\tfrac{1}{\brho_{\max}^{2}}\int_{\R}\partial_{2}\bW^{2}_{\eta}(t,y)H\big(\bW^{1}_{\eta}(t,y),\bW^{2}_{\eta}(t,y),y\big)\sgn(\partial_{y}\bW^{1}_{\eta}(t,y))\dd y\\
   &\quad+ \tfrac{1}{\brho_{\max}^{2}}\int_{\R}\partial_{2}\bW^{2}_{\eta}(t,y)\tfrac{\dd}{\dd y}H\big(\bW^{1}_{\eta}(t,y),\bW^{2}_{\eta}(t,y),y\big) \int_{-\infty}^{y}\sgn(\partial_{x}\bW^{1}_{\eta}(t,x))\exp\big(\tfrac{x-y}{\eta}\big)\dd x\dd y\\
    &\quad +\tfrac{1}{\eta\brho_{\max}^{1}}\|H\|_{L^{\infty}((0,|\brho_{\max}|_{\infty})\times(0,|\brho_{\max}|_{\infty})\times\R)}\int_{\R}|\partial_{2}\bW^{1}_{\eta}(t,y)|\int_{-\infty}^{y}\exp\big(\tfrac{x-y}{\eta}\big)\dd x\dd y\\
    &\quad +\tfrac{1}{\brho_{\max}^{1}}\lim_{y\rightarrow\infty}\partial_{2}\bW^{1}_{\eta}(t,y)H\big(\bW^{1}_{\eta}(t,y),\bW^{2}_{\eta}(t,y),y\big) \int_{-\infty}^{y}\sgn(\partial_{x}\bW^{1}_{\eta}(t,x))\exp\big(\tfrac{x-y}{\eta}\big)\dd x\dd y\\
    &\quad -\tfrac{1}{\brho_{\max}^{1}}\int_{\R}\partial_{2}\bW^{1}_{\eta}(t,y)H\big(\bW^{1}_{\eta}(t,y),\bW^{2}_{\eta}(t,y),y\big) \sgn(\partial_{y}\bW^{1}_{\eta}(t,y))\dd y\\
    &\quad -\tfrac{1}{\brho_{\max}^{1}}\int_{\R}\partial_{2}\bW^{1}_{\eta}(t,y)\tfrac{\dd}{\dd y}H\big(\bW^{1}_{\eta}(t,y),\bW^{2}_{\eta}(t,y),y\big) \int_{-\infty}^{y}\sgn(\partial_{x}\bW^{1}_{\eta}(t,x))\exp\big(\tfrac{x-y}{\eta}\big)\dd x\dd y\\
    &\quad +2\|\partial_{1}H\|_{L^{\infty}((0,|\brho_{\max}|_{\infty})\times(0,|\brho_{\max}|_{\infty})\times\R)}\big|\bW^{1}_{\eta}(t,\cdot)\big|_{TV(\R)}\\
    &\quad +2\|\partial_{2}H\|_{L^{\infty}((0,|\brho_{\max}|_{\infty})\times(0,|\brho_{\max}|_{\infty})\times\R)}\big|\bW^{2}_{\eta}(t,\cdot)\big|_{TV(\R)}\\
    &\quad +\|H\|_{L^{\infty}((0,|\brho_{\max}|_{\infty})\times(0,|\brho_{\max}|_{\infty});TV(\R))}
    \intertext{applying \cref{lem:W_vanishing_infty}, i.e.,\ \(\lim_{y\rightarrow\infty}\partial_{2}\bW_{\eta}^{i}t,y)=0,\ \forall t\in[0,T],\ i\in\{1,2\}\) and recalling the postulated bounds on \(H\) in \cref{ass:general}} 
     &\leq2\tfrac{\mathcal{H}}{\brho_{\max}^{2}}|\bW^{2}_{\eta}(t,\cdot)|_{TV(\R)} + \tfrac{\eta}{\brho_{\max}^{2}}\int_{\R}\big|\partial_{2}\bW^{2}_{\eta}(t,y)\big|\big|\tfrac{\dd}{\dd y}H\big(\bW^{1}_{\eta}(t,y),\bW^{2}_{\eta}(t,y),y\big)\big|\dd y\\
    &\quad+2\tfrac{\mathcal{H}}{\brho_{\max}^{1}}\big|\bW^{1}_{\eta}(t,\cdot)\big|_{TV(\R)} +\tfrac{\eta}{\brho_{\max}^{1}}\int_{\R}\big|\partial_{2}\bW^{1}_{\eta}(t,y)\big|\big|\tfrac{\dd}{\dd y}H\big(\bW^{1}_{\eta}(t,y),\bW^{2}_{\eta}(t,y),y\big)\big|\dd y\\
    &\quad +2\mathcal{H}_{1}\big|\bW^{1}_{\eta}(t,\cdot)\big|_{TV(\R)}+2\mathcal{H}_{2}\big|\bW^{2}_{\eta}(t,\cdot)\big|_{TV(\R)} +\mathcal{H}_{BV}
    \intertext{and taking advantage of \cref{eq:identity_nonlocal}, and in particular \(\eta\partial_{2}\bW^{i}(t,x)=\bW_{\eta}^{i}(t,x)-\brho^{i}(t,x)\ \implies \eta\|\partial_{2}\bW^{i}(t,\cdot)\|_{L^{\infty}(\R)}\leq 2\|\bq_{\max}\|_{\infty}\ \forall (t,x)\in\OT\)}
    &\leq2\tfrac{\mathcal{H}}{\brho_{\max}^{2}}|\bW^{2}_{\eta}(t,\cdot)|_{TV(\R)} + 2\tfrac{\|\brho_{\max}\|_{\infty}}{\brho_{\max}^{2}}\int_{\R}\big|\tfrac{\dd}{\dd y}H\big(\bW^{1}_{\eta}(t,y),\bW^{2}_{\eta}(t,y),y\big)\big|\dd y\\
    &\quad+2\tfrac{\mathcal{H}}{\brho_{\max}^{1}}\big|\bW^{1}_{\eta}(t,\cdot)\big|_{TV(\R)} +2\tfrac{\|\brho_{\max}\|_{\infty}}{\brho_{\max}^{1}}\int_{\R}\big|\tfrac{\dd}{\dd y}H\big(\bW^{1}_{\eta}(t,y),\bW^{2}_{\eta}(t,y),y\big)\big|\dd y\\
    &\quad +2\mathcal{H}_{1}\big|\bW^{1}_{\eta}(t,\cdot)\big|_{TV(\R)}+2\mathcal{H}_{2}\big|\bW^{2}_{\eta}(t,\cdot)\big|_{TV(\R)} +\mathcal{H}_{BV}\\
    &\leq 2\tfrac{\mathcal{H}}{\brho_{\max}^{2}}|\bW^{2}_{\eta}(t,\cdot)|_{TV(\R)} + 2\tfrac{\|\brho_{\max}\|_{\infty}}{\brho_{\max}^{2}}\big(\mathcal{H}_{1}\big|\bW^{1}_{\eta}(t,\cdot)\big|_{TV(\R)}+\mathcal{H}_{2}\big|\bW^{2}_{\eta}(t,\cdot)\big|_{TV(\R)}+\mathcal{H}_{BV}\big)\\
    &\quad+2\tfrac{\mathcal{H}}{\brho_{\max}^{1}}\big|\bW^{1}_{\eta}(t,\cdot)\big|_{TV(\R)} +2\tfrac{\|\brho_{\max}\|_{\infty}}{\brho_{\max}^{1}}\big(\mathcal{H}_{1}\big|\bW^{1}_{\eta}(t,\cdot)\big|_{TV(\R)}+\mathcal{H}_{2}\big|\bW^{2}_{\eta}(t,\cdot)\big|_{TV(\R)}+\mathcal{H}_{BV}\big)\\
    &\quad +2\mathcal{H}_{1}\big|\bW^{1}_{\eta}(t,\cdot)\big|_{TV(\R)}+2\mathcal{H}_{2}\big|\bW^{2}_{\eta}(t,\cdot)\big|_{TV(\R)} +\mathcal{H}_{BV}\\
    &=2\Big(\tfrac{\|\brho_{\max}\|_{\infty}}{\brho_{\max}^{2}}\mathcal{H}_{1}+\tfrac{\mathcal{H}}{\brho_{\max}^{1}}+\tfrac{\|\brho_{\max}\|_{\infty}}{\brho_{\max}^{1}}\mathcal{H}_{1}+\mathcal{H}_{1}\Big)\big|\bW^{1}_{\eta}(t,\cdot)\big|_{TV(\R)}\\
    &\quad +2\Big(\tfrac{\|\brho_{\max}\|_{\infty}}{\brho_{\max}^{1}}\mathcal{H}_{1}+\tfrac{\mathcal{H}}{\brho_{\max}^{2}}+\tfrac{\|\brho_{\max}\|_{\infty}}{\brho_{\max}^{2}}\mathcal{H}_{1}+\mathcal{H}_{1}\Big)\big|\bW^{2}_{\eta}(t,\cdot)\big|_{TV(\R)}\\
    &\quad +2\Big(\tfrac{\|\brho_{\max}\|_{\infty}}{\brho_{\max}^{2}}+\tfrac{\|\brho_{\max}\|_{\infty}}{\brho_{\max}^{1}}+1 \Big) \mathcal{H}_{BV}.
     \end{align*} 
    In a similar manner, we can derive the (almost) identical estimate for the change in time of the total variation of \(\bW_{2}\), leading us to the estimate
    \begin{gather*}
\tfrac{\dd}{\dd t}\Big(\big|\bW^1_{\eta}(t,\cdot)\big|_{TV(\R)}+\big|\bW^2_{\eta}(t,\cdot)\big|_{TV(\R)}\Big)=\tfrac{\dd}{\dd t} \big|\bW_{\eta}(t,\cdot)\big|_{TV(\R;\R^{2})}\\
\leq 2 \Big(\tfrac{\|\brho_{\max}\|_{\infty}}{\brho_{\max}^{2}}\mathcal{H}_{1}+\tfrac{\mathcal{H}}{\brho_{\max}^{1}}+\tfrac{\|\brho_{\max}\|_{\infty}}{\brho_{\max}^{1}}\mathcal{H}_{1}+2\mathcal{H}_{1}+\tfrac{\|\brho_{\max}\|_{\infty}}{\brho_{\max}^{1}}\mathcal{H}_{1}+\tfrac{\mathcal{H}}{\brho_{\max}^{2}}+\tfrac{\|\brho_{\max}\|_{\infty}}{\brho_{\max}^{2}}\mathcal{H}_{1}\Big)\big|\bW_{\eta}(t,\cdot)\big|_{TV(\R;\R^{2})}\\
\qquad +4\Big(\tfrac{\|\brho_{\max}\|_{\infty}}{\brho_{\max}^{2}}+\tfrac{\|\brho_{\max}\|_{\infty}}{\brho_{\max}^{1}}+1 \Big) \mathcal{H}_{BV}.
    \end{gather*}
Using Gronwall's inequality~\cite{Dragomir2003} yields: 
\begin{align*}
    |\bW_{\eta}(t,\cdot)|_{TV(\R;\R^{2})}&\leq \bigg(\big|\bW_{\eta}(0,\cdot)\big|_{TV(\R;\R^{2})}+4\Big(\tfrac{\|\brho_{\max}\|_{\infty}}{\brho_{\max}^{2}}+\tfrac{\|\brho_{\max}\|_{\infty}}{\brho_{\max}^{1}}+1 \Big) \mathcal{H}_{BV}\bigg)\\
    &\ \cdot\exp\bigg(2t \Big(\tfrac{\|\brho_{\max}\|_{\infty}}{\brho_{\max}^{2}}\mathcal{H}_{1}\!+\tfrac{\mathcal{H}}{\brho_{\max}^{1}}+\tfrac{\|\brho_{\max}\|_{\infty}}{\brho_{\max}^{1}}\mathcal{H}_{1}+2\mathcal{H}_{1}+\tfrac{\|\brho_{\max}\|_{\infty}}{\brho_{\max}^{1}}\mathcal{H}_{1}+\tfrac{\mathcal{H}}{\brho_{\max}^{2}}+\tfrac{\|\brho_{\max}\|_{\infty}}{\brho_{\max}^{2}}\mathcal{H}_{1}\Big)\bigg).
\end{align*}
As this estimate is uniform in the approximation, and it holds

\[
\big|\bW_{\eta}(0,\cdot)\big|_{TV(\R;\R^{2})}\leq |\bq_{0}|_{TV(\R;\R^{2})},
\]
we obtain the uniform \(TV\) bound for any initial datum of given \(TV\) regularity.
\end{proof}
\begin{remark}[Consistency with the \(TV\) estimate for nonlocal conservation laws]
Assuming there is no lane change, i.e.,\ \(S\equiv 0\), the total variation estimate derived in \cref{thm:tv_bound} reduces to:
\begin{equation}
\big|\bW_{\eta}(t,\cdot)\big|_{TV(\R;\R^{2})}\leq |\bq_{0}|_{TV(\R;\R^{2})}\ \forall t\in[0,T].\label{eq:TV_diminishing}
\end{equation}
Thus, the nonlocal term exhibits total variation diminishing behavior. This observation is not surprising because there is no coupling between the two nonlocal equations in this case. Consequently, we are dealing with the singular limit problem for scalar nonlocal conservation laws for which an estimate/bound similar to\cref{eq:TV_diminishing} was obtained in \cite[Theorem 3.2]{coclite2020general}.
\end{remark}

\subsection{Entropy admissibility}
In this section, we demonstrate that, given strong convergence, the solutions to the nonlocal system are entropy-admissible in the limit. The approach parallels the strategies outlined in\cite{bressan2021entropy,Marconi2023}:
\begin{theorem}[Entropy admissibility]\label{thm:entropy_admissibility}
Let \(\brho_{\eta}\in C\big([0,T];L^{1}_{\text{loc}}(\R;\R^{2})\big)\cap L^{\infty}\big((0,T);L^{\infty}(\R;\R^{2})\big)\) be the unique solution of \cref{eq:nonlocal_system}. Assume that there exists \(\brho^{*}\in C\big([0,T];L^{1}_{\text{loc}}(\R;\R^{2})\big)\cap L^{\infty}\big((0,T);L^{\infty}(\R;\R^{2})\big)\) such that
\[
\lim_{\eta\rightarrow 0}\|\brho_{\eta}-\brho^{*}\|_{C([0,T];L^{1}_{\text{loc}}(\R;\R^2))}=0,\qquad \exists C\in\R_{>0}:\ \sup_{\eta\in\R_{>0}}|\cW_{\eta}[\brho_{\eta}]|_{L^{\infty}((0,T);TV(\R;\R^{2}))}\leq C.\]
   Then, \(\brho^{*}\) satisfies the entropy admissibility condition in \cref{defi:entropy} for a general convex entropy \(\alpha''(x)\geq 0, \, \beta'(x)=\alpha'(x)[V(x)+x V'(x)]\).
\end{theorem}
\begin{proof}
Let us define \((\alpha, \beta), \, \alpha,\beta\in C^{2}(\R;\R)\)  such that \(\alpha''(x)\geq 0,\, \beta'(x)=\alpha'(x)[V(x)+x V'(x)]\). We also fix \(0\leq \varphi \in C^\infty_c(\Omega_T)\). Our goal is to prove that 
\begin{align*}
\mEF_1[\phi,\alpha,\brho^1_\ast]&\:\iint_{\OT}\alpha(\brho_\ast^1(t,x))\phi_{t}(t,x)+\beta_1(\brho_\ast^1(t,x))\phi_{x}(t,x)\dd x\dd t+\int_{\R}\alpha(\brho^1_{0}(x))\phi(0,x)\dd x\\
&\quad-\iint_{\OT}\alpha'(\brho^1_\ast(t,x)) S\big(\brho_\ast^{1}(t,x),\brho_\ast^{2}(t,x),\brho_\ast^{1}(t,x),\brho_\ast^{2}(t,x),x\big)\phi(t,x)\dd x\dd t \geq 0,\\
\mEF_2[\phi,\alpha,\brho^2_\ast]&\:\iint_{\OT}\alpha(\brho^2_\ast(t,x))\phi_{t}(t,x)+\beta_2(\brho^2_\ast(t,x))\phi_{x}(t,x)\dd x\dd t+\int_{\R}\alpha(\brho^2_{0}(x))\phi(0,x)\dd x\\
&\quad+\iint_{\OT}\alpha'(\brho_\ast^2(t,x)) S\big(\brho_\ast^{1}(t,x),\brho_\ast^{2}(t,x),\brho_\ast^{1}(t,x),\brho_\ast^{2}(t,x),x\big)\phi(t,x)\dd x\dd t \geq 0.
\end{align*}
We choose a sequence \(\eta_k\), which is still denoted by \(\eta\) and set \(\bW^i_{\eta}\:\tfrac{1}{\eta}\int_{x}^{\infty}\exp\big(\tfrac{x-y}{\eta}\big)\brho^{i}(t,y)\dd y\). Then, we set 
\begin{equation}
\begin{aligned}
\mEF_1[\phi,\alpha,\bW^1_{\eta}]&\:\iint_{\OT}\alpha(\bW^1_{\eta})\phi_{t}(t,x)+\beta_1(\bW^1_{\eta})\phi_{x}(t,x)\dd x\dd t+\int_{\R}\alpha(\bW^1_{\eta} (0,x))\phi(0,x)\dd x,\\
&\quad-\iint_{\OT}\alpha'(\bW^1_{\eta}) S\big(\bW_{\eta}, \bW_{\eta},x\big)\phi(t,x)\dd x\dd t \\
\mEF_2[\phi,\alpha,\bW^2_{\eta}]&\:\iint_{\OT}\alpha(\bW^2_{\eta})\phi_{t}(t,x)+\beta_2(\bW^2_{\eta})\phi_{x}(t,x)\dd x\dd t+\int_{\R}\alpha(\bW^2_{\eta}(0,x))\phi(0,x)\dd x\\
&\quad+\iint_{\OT}\alpha'(\bW^2_{\eta}) S\big(\bW_{\eta}, \bW_{\eta},x\big)\phi(t,x)\dd x\dd t.
\end{aligned}
\label{eq:entropy_system_nonlocal_term}
\end{equation}
We recall that by assumption \(\bW^i_{\eta}\to \brho^i\) in \(L^1_{\text{loc}}(\Omega_T)\) and that \[\lim_{k\to+\infty}\mEF_{i}[\phi,\alpha,\bW_{\eta_{k}}^{i}]=\mEF_i[\phi,\alpha,\brho_{*}^{i}].\]
Hence, we need to show: \begin{equation}\label{eq:aim}
    \lim_{k\to \infty} \mEF_{i}[\phi,\alpha,\bW_{\eta}^{i}] \geq 0\quad \forall \phi\in C^{\infty}_{\text{c}}(\OT;\R_{\geq0}),\quad \forall \alpha\in C^{2}(\R)\ \text{convex},\ \forall i\in\{1,2\}.
\end{equation} 
For simplicity, we use the notation \(\brho \ast \exp_\eta\coloneqq \tfrac{1}{\eta}\int_x^{\infty} \exp\big(\tfrac{x-y}{\eta}\big) \brho(t,y) \dd y
\),
First, we rewrite \(\mEF_i,\ i\in\{1,2\}\) and obtain, suppressing the subsequence index for \(\eta\in\R_{>0}\), 
  \begin{gather*}
\iint_{\OT}  \alpha(\bW^i_{\eta}) \partial_t \varphi + \left[\left(V(\bW^i_{\eta})\brho_{\eta}^i\right)\ast \exp_{\eta}\right]\partial_x \left[\alpha'(\bW^i_{\eta}) \varphi \right] \dd x\dd t+\int_\R \alpha\big(\bW^i_{\eta}(0,x)\big) \varphi(0,x) \dd x\\
=(-1)^{i+1} \iint_{\OT} \alpha'(\bW^i_{\eta}) \mathscr{S}\big(\bW_{\eta},\eta\partial_{x}\bW_{\eta}(t,x),x\big) \ast \exp_{\eta} \varphi(t,x) \dd x \dd t
    \end{gather*}  
    for \(\mathscr{S}\), as reported in \cref{eq:abbreviation_S}.
    Thanks to the equality \(\beta'_{i}(x)=\alpha'(x)\left[V(x)+x V'(x)\right],\ x\in\R\), we obtain  
    \begin{align*}
        \iint_{\OT} \beta_{i}(\bW^i_{\eta}) \partial_x \varphi \dd x \dd t&=-\iint_{\OT}\beta'_{i}(\bW^i_{\eta})\partial_{x}\bW^i_{\eta}\varphi\dd x\dd t\\
        &=-\iint_{\OT}\alpha'(\bW^{i}_{\eta})V(\bW^{i}_{\eta})\partial_{x}\bW^i_{\eta}\varphi\dd x\dd t-\iint_{\OT}\alpha'(\bW^{i}_{\eta})V'(\bW^{i}_{\eta})\bW^{i}_{\eta}\partial_{x}\bW^i_{\eta}\varphi\dd x\dd t
        \intertext{and integration by parts in the last term leads to (interpreting \(\tfrac{\dd}{\dd x}V(\bW^{i}_{\eta})=V'(\bW^{i}_{\eta})\partial_{x}\bW^{i}_{\eta}\))}
        &= \iint_{\OT} V(\bW^i_{\eta}) \bW^i_{\eta} \partial_x [\alpha'(\bW^i_{\eta})\varphi] \dd x \dd t.
    \end{align*}
    Then, by referencing\cref{eq:entropy_system_nonlocal_term} for \(i\in\{1,2\}\), we obtain the following:
    \begin{align}
        &\mEF_{i}[\phi,\alpha,\bW_{\eta}^{i}]\\
        &= \iint_{\OT} \left[V(\bW_\eta^{i}) \bW_i^{\eta} - \left(V(\bW_\eta^{i})\brho_{\eta}^i\right)\ast \exp_{\eta}\right] \partial_x [\alpha'(\bW_\eta^{i})\varphi] \dd x \dd t \\
        &\quad (-1)^{i+1} \iint_{\OT} \alpha'(\bW_\eta^{i}) \left[S\big(\brho_{\eta},\bW_{\eta},x\big) \ast \exp_{\eta}\right] \varphi(t,x) \dd x \dd t \\
        &\quad (-1)^{i} \iint_{\OT}\alpha'(\bW_\eta^{i}) S\big(\bW_{\eta},\bW_{\eta}, x\big)\phi(t,x)\dd x\dd t \\
        &= \iint_{\OT} \left[V(\bW_\eta^{i}) \bW_\eta^{i} - \left(V(\bW_\eta^{i})\brho_{\eta}^i\right)\ast \exp_{\eta}\right] \partial_x [\alpha'(\bW_\eta^{i})\varphi] \dd x \dd t 
        \\
        &\quad (-1)^{i} \iint_{\OT} \alpha'(\bW_\eta^{i}) 
        \left(\int_x^{+\infty} \exp\big(\tfrac{x-y}{\eta}\big)\left[-S\big(\brho_{\eta}(y,t),\bW_{\eta}(y,t),y\big) +S\big(\bW_{\eta}(t,x), \bW_{\eta}(x,t), x\big)\right] \dd y\right)\varphi(t,x) \dd x \dd t\\
     &=\iint_{\OT} \left[V(\bW_\eta^{i}) \bW_\eta^{i} - \left(V(\bW^i_{\eta})\brho_{\eta}^i\right)\ast \exp_{\eta}\right] \partial_x [\alpha'(\bW_\eta^{i})] \varphi\dd x \dd t\label{eq:two_terms_before_last_term}\\
     &\quad+\iint_{\OT} \left[V(\bW_\eta^{i}) \bW_\eta^{i} - \left(V(\bW_\eta^{i})\brho_{\eta}^i\right)\ast \exp_{\eta}\right] \alpha'(\bW_\eta^{i}) \partial_x \varphi\dd x \dd t \label{eq:one_term_before_last_term}\\
        &\quad(-1)^{i} \iint_{\OT}\!\! \alpha'(\bW_\eta^{i}) 
        \left(\int_x^{+\infty} \!\!\!\!\!\!\!\!\exp\big(\tfrac{x-y}{\eta}\big)\left[-S\big(\brho_{\eta}(t,y),\bW_{\eta}(t,y),y\big) +S\big(\bW_{\eta}(t,x),\bW_{\eta}(t,x), x\big)\right] \dd y\right)\varphi(t,x) \dd x \dd t. \label{eq:last_term}
    \end{align} 
    Note that the second term in the previous equality converges to zero for $\eta\to 0$:
    \begin{align*}
|\eqref{eq:one_term_before_last_term}|&\leq \iint_{\OT}\Big| V(\bW_\eta^{i}) \bW_\eta^{i} - \tfrac{1}{\eta}\int_{x}^{\infty}\exp\big(\tfrac{x-y}{\eta}\big)V(\bW_\eta^{i}(t,y))\brho_{\eta}^i(t,y)\dd y\Big| \big|\alpha'(\bW_\eta^{i}) \partial_x \varphi\dd x\big| \dd t \\
&\overset{\eqref{eq:identity_nonlocal}}{=}
\iint_{\OT}\!\!\Big| V(\bW_\eta^{i}) \bW_\eta^{i} - \tfrac{1}{\eta}\int_{x}^{\infty}\!\!\!\!\exp\big(\tfrac{x-y}{\eta}\big)V(\bW_\eta^{i}(t,y))\big(\bW_{\eta}^{i}(t,y)-\eta\partial_{y}\bW_{\eta}^{i}(t,y)\big)\dd y\Big| \big|\alpha'(\bW_\eta^{i}) \partial_x \varphi\dd x\big| \dd t 
\intertext{and splitting the difference in the sum and performing integration by parts  yields}
&\leq \|\alpha'\|_{L^{\infty}((0,\|\brho_{\max}\|_{\infty}))}\|\partial_{2}\phi\|_{L^{\infty}(\OT)}\iint_{\OT}\Big| -\int_{x}^{\infty}\!\!\exp\big(\tfrac{x-y}{\eta}\big)V'(\bW_\eta^{i}(t,y))\bW_{\eta}^{i}(t,y)\partial_{y}\bW_{\eta}^{i}(t,y)\dd y\Big|\dd x \dd t 
 \intertext{and a change of order of integration}
 &= \|\alpha'\|_{L^{\infty}((0,\|\brho_{\max}\|_{\infty}))}\|\partial_{2}\phi\|_{L^{\infty}(\OT)}\int_{0}^{T}\int_{\R}\Big|V'(\bW_\eta^{i}(t,y))\bW_{\eta}^{i}(t,y)\partial_{y}\bW_{\eta}^{i}(t,y)\Big| \int_{-\infty}^{y}\exp\big(\tfrac{x-y}{\eta}\big)\dd x\dd y\\
 &= \eta\|\alpha'\|_{L^{\infty}((0,\|\brho_{\max}\|_{\infty}))}\|\partial_{2}\phi\|_{L^{\infty}(\OT)}\int_{0}^{T}\int_{\R}\Big|V'(\bW_\eta^{i}(t,y))\bW_{\eta}^{i}(t,y)\partial_{y}\bW_{\eta}^{i}(t,y)\Big|\dd y\\
 &\leq\eta\|\alpha'\|_{L^{\infty}((0,\|\brho_{\max}\|_{\infty}))}\|\partial_{2}\phi\|_{L^{\infty}(\OT)}T\|V'\|_{L^{\infty}((0,\|\brho_{\max}\|_{\infty}))}\brho^{i}_{\max}|\bW_{\eta}^{i}|_{L^{\infty}((0,T);TV(\R))}
 \end{align*}
    The last term is bounded by assumption and converges to zero for \(\eta\rightarrow 0\), as claimed.

    The third term cancels out because, practically speaking, \(S\) and \(\alpha'\) are bounded, and \(\phi\) has compact support. Consequently, the
    integration in the exponential kernel yields the following (recalling the assumptions on the lane-changing in \cref{ass:general}):
    \begin{align*}
        |\eqref{eq:last_term}|&\leq\|\alpha'\|_{L^{\infty}((0,\|\brho_{\max}\|_{\infty}))}2\|\brho_{\max}\|_{\infty} \iint_{\OT}|\phi(t,x)|\int_{x}^{\infty}\exp\big(\tfrac{x-y}{\eta}\big)H(\bW_{\eta},y)\dd y\dd x\\
        &\leq \mathcal{H}\|\alpha'\|_{L^{\infty}((0,\|\brho_{\max}\|_{\infty}))}2\|\brho_{\max}\|_{\infty}\iint_{\OT}|\phi(t,x)|\int_{x}^{\infty}\exp\big(\tfrac{x-y}{\eta}\big)\dd y\dd x\\
        &\leq \eta\mathcal{H}\|\alpha'\|_{L^{\infty}((0,\|\brho_{\max}\|_{\infty}))}2\|\brho_{\max}\|_{\infty}\|\phi\|_{L^{\infty}(\OT)}\supp(\phi)
    \end{align*}
    which converges to zero for \(\eta\rightarrow 0\).
    Hence, the only term left needed to treat is the term in \eqref{eq:two_terms_before_last_term}. To accomplish this,  we defined 
    \begin{align*}
    T^{\eta}_1\coloneqq\iint_{\OT} \left[V(\bW_\eta^{i}) \bW_\eta^{i} - \left(V(\bW_\eta^{i})\brho_{\eta}^i\right)\ast \exp_{\eta}\right]  \partial_x [\alpha'(\bW_\eta^{i})]\varphi\dd x \dd t, 
    \end{align*}
     so we can write: 
    \begin{align*}
    T^{\eta}_1 &=\iint_{\OT} \int_x^{+\infty} \left[V(\bW_\eta^{i}(t,x))  - \left(V(\bW_\eta^{i}(t,y)\right) \right]  \partial_x [\alpha'(\bW_i^{\eta})](t,x)\varphi(t,x) \tfrac{1}{\eta} \exp\left(\tfrac{x-y}{\eta}\right) \brho^i_\eta(t,y) \dd y \dd x \dd t\\
    &=\iint_{\Omega_T} \brho^i_\eta (t,y) \omega^i_\eta(t,y) \dd y \dd t,
    \end{align*}
    where 
    \begin{equation}
    \label{eq:omega1}
    \begin{aligned}
        \omega^i_\eta(t,y)&\coloneqq \int_{-\infty}^y  \left[V(\bW_\eta^{i}(t,x))  - \left(V(\bW_\eta^{i}(t,y)\right) \right]  \partial_x [\alpha'(\bW_i^{\eta})](t,x)\varphi(t,x) \tfrac{1}{\eta} \exp\left(\tfrac{x-y}{\eta}\right)  \dd x\\
        &=\int_{-\infty}^y \underbrace{V(\bW_\eta^{i}(t,x))  \partial_x [\alpha'(\bW_i^{\eta})](t,x)}_{\eqqcolon\partial_x I(\bW^i_\eta)}\varphi(t,x) \tfrac{1}{\eta} \exp\left(\tfrac{x-y}{\eta}\right) \dd x\\
        &-V(\bW_\eta^{i}(t,y)) \int_{-\infty}^y \partial_x [\alpha'(\bW_i^{\eta})](t,x)\varphi(t,x) \tfrac{1}{\eta} \exp\left(\tfrac{x-y}{\eta}\right) \dd x.
    \end{aligned} 
    \end{equation}
    Using partial integration, we obtain 
    \begin{equation}
    \label{eq:omega}
    \begin{aligned}
        \omega^i_\eta(t,y)&=\frac{1}{\eta} I(\bW^i_\eta(t,y)) \varphi(t,y)-\int_{-\infty}^y I(\bW^i_\eta (t,x)) \partial_x \left[\varphi(t,x) \tfrac{1}{\eta} \exp\left(\tfrac{x-y}{\eta}\right)\right] \dd x\\
        &-V(\bW^i_\eta(t,y)) \left[\alpha'(\bW^i_\eta(t,y)) \varphi(t,y)\tfrac{1}{\eta}-\int_{-\infty}^y
        \alpha'(\bW^i_\eta(t,x)) \partial_x \left[\varphi(t,x)\tfrac{1}{\eta}\exp\left(\tfrac{x-y}{\eta}\right)\right] \dd x 
        \right]\\
        &=\int_{-\infty}^y [I(\bW^i_\eta(t,y))-I(\bW^i_\eta(t,x))] \partial_x \left[
        \varphi(t,x) \tfrac{1}{\eta} \exp\left(\tfrac{x-y}{\eta}\right)
        \right] \dd x\\
        &-V(\bW^i_\eta(t,y)) \int_{-\infty}^y [\alpha'(\bW^i_\eta(t,y))-\alpha'(\bW^i_\eta(t,x))] 
        \partial_x \left[\varphi(t,x) \tfrac{1}{\eta} \exp\left(\tfrac{x-y}{\eta}\right)\right] \dd x\\
        &=G_{\eta}(t,y)+L_{\eta}(t,y)+P_\eta(t,y),
    \end{aligned}
    \end{equation}
    with 
    \begin{align}
    \label{eq:G}
      G_{\eta}(t,y)&\coloneqq  \int_{-\infty}^y [I(\bW^i_\eta(t,y))-I(\bW^i_\eta(t,x))]  \left[
         \tfrac{1}{\eta} \exp\left(\tfrac{x-y}{\eta}\right)
        \right] \partial_x \varphi(t,x) \dd x,\\
        \label{eq:L}
        L_{\eta}(t,y)&\coloneqq -V(\bW^i_\eta(t,y)) \int_{-\infty}^y [\alpha'(\bW^i_\eta(t,y))-\alpha'(\bW^i_\eta(t,x))] 
         \left[ \tfrac{1}{\eta} \exp\left(\tfrac{x-y}{\eta}\right)\right] \partial_x \varphi(t,x) \dd x, 
    \intertext{and}
       P_\eta(t,y)&\coloneqq\int_{-\infty}^{y} H(\bW^i_\eta(t,x),\bW^i_\eta(t,y)) \varphi(t,x) \partial_x \left[\tfrac{1}{\eta}\exp\left(\tfrac{x-y}{\eta}\right)\right] \dd x\\
        \label{eq:P}
        &=\tfrac{1}{\eta^2} \int_{-\infty}^{y} H(\bW^i_\eta(t,x),\bW^i_\eta(t,y)) \varphi(t,x)  \exp\left(\tfrac{x-y}{\eta}\right) \dd x,
    \end{align}
    where 
    \begin{equation*}
        H(a,b)\coloneqq I(b)-I(a)-V(b)(\alpha'(b)-\alpha'(a)).
    \end{equation*}
    Next, by plugging \cref{eq:G}, \cref{eq:L} and \cref{eq:P} into \cref{eq:omega}, we can formulate:
    \begin{align*}
       \mEF_{i}[\phi,\alpha,\bW_{\eta}^{i}]&\geq 
       \iint_{\Omega_T} \brho_\eta^i(t,y) \left[G_\eta(t,y)+L_\eta(t,y)+P_\eta(t,y)\right] \dd y \dd t.
    \end{align*} 
    We now can show that 
    \begin{equation}
        \begin{aligned}\label{eq:claim1}
        \mEF_{i}[\phi,\alpha,\bW_{\eta}^{i}] &\geq 
          \iint_{\Omega_T} \brho_\eta^i(t,y) \left[G_\eta(t,y)+L_\eta(t,y)\right] \dd y \dd t.
 \end{aligned}
 \end{equation}
 It is sufficient to prove that $P_\eta \geq 0.$
 To accomplish this, we compute 
 \begin{equation*}
     \frac{\partial H}{\partial a}(u,b)=-I'(u)+V(b) \alpha''(u)=\alpha''(u)[V(b)-V(u)] 
 \end{equation*}
 and apply the same argument as in \cite[Proof of Theorem 1.2]{Marconi2023}, so it can be concluded that $P_\eta \geq 0.$  
 To establish \cref{eq:aim}, it suffices to show that the right-hand side of \cref{eq:claim1} vanishes for $\eta\to 0.$ 
 We now can show that 
 \begin{equation*}
     \lim_{\eta\to 0} \iint_{\Omega_T} \brho_\eta^i(t,y) G_\eta(t,y) \dd y \dd t=0.
 \end{equation*}
 To achieve this, we can write the following:
 \begin{align*}
     &\iint_{\Omega_T} \brho_\eta^i(t,y) G_\eta(t,y) \dd y \dd t\\
     &\leq 
     \iint_{\Omega_T} \brho_\eta^i(t,y) \int_{-\infty}^y |I(\bW_\eta^{i}(t,y)))-I(\bW_\eta^{i}(t,x)))| \tfrac{1}{\eta} \exp\left(\tfrac{x-y}{\eta}\right) |\partial_x\varphi(t,x)| \dd x \dd y \dd t\\
     &\overset{\text{Fubini}}{=} \iint_{\Omega_T} |\partial_x \varphi(t,x)| \int_x^{+\infty} \brho^i_\eta (t,y) |I(\bW_\eta^{i}(t,y)))-I(\bW_\eta^{i}(t,x)))| \tfrac{1}{\eta} \exp\left(\tfrac{x-y}{\eta}\right) \dd y \dd x \dd t.
 \end{align*}
 Because $\varphi$ is compactly supported by applying~\cite[Lemma 4.1]{Marconi2023}, we can conclude that it vanishes for $\eta\to 0.$ 
 
 Analogously, one can show that 
\begin{equation*}
     \lim_{\eta\to 0} \iint_{\Omega_T} \brho_\eta^i(t,y) L_\eta(t,y) \dd y \dd t=0,
 \end{equation*}
which concludes the proof.
 
\end{proof}

\subsection{Main Theorem and some Corollaries}
So far, we have proven entropy admissibility in \cref{thm:entropy_admissibility} and for the nonlocal operator a \(TV\) bound uniform in \(\eta\in\R_{>0}\) in \cref{thm:tv_bound}. However, this \(TV\) bound is only in space, and to obtain compactness in \(C(L^{1})\), a ``time-compactness'' is required as well. This is what is established in the next theorem:

\begin{theorem}[Compactness of $\bW_\eta$]\label{thm:compactness}
The set of nonlocal terms $(\bW_\eta)_{\eta\in \R_{>0}}\subseteq C\big([0,T];L^{1}_{\loc}(\R;\R^{2})\big)$ of solutions to \cref{eq:W_system} is compactly embedded into $C\big([0,T];L^{1}_{\loc}(\R;\R^{2})\big),$ i.e.,
\[
\{\bW_{\eta},\ \eta\in\R_{>0}\}\overset{\text{c}}{\hookrightarrow}C\big([0,T];L^{1}_{\loc}(\R;\R^{2})\big)
\]
\end{theorem}
\begin{proof}
We now apply~\cite[Lemma 1]{Simon1986}. In particular, according to the notation in~\cite[Lemma 1]{Simon1986}, we set the Banach space \(B=L^1_{\text{loc}}(\Omega)\) with \(\Omega\subset\R\) open bounded and for \(t\in[0,T]\)
\[F(t) \coloneqq\big\{\bW_\eta(t,\cdot)\in L^1_{\text{loc}}(\R),\quad \eta \in \R_{>0}\big\}.\] 
According to ~\cite[Theorem 13.35]{leoni}, the set \(F(t)\) is compact in \(L^1_{\text{loc}}(\R)\) because of the total uniform variation bound in the spatial component of \(\bW_\eta\) proved in \cref{thm:tv_bound}. Moreover, the set \((\bW_\eta)_{\eta\in \R_{>0}}\) is uniformly equi-continuous. To accomplish this, we estimate for \((t_{1},t_{2})\in[0,T]\) (assuming we have regular enough solutions, that we can assume thanks to \cref{lemma:stability})
\begin{align*}
   & \|\bW^{1}_{\eta}(t_1,\cdot)-\bW^{1}_{\eta}(t_2,\cdot)\|_{L^1(\Omega)}=\Big\| \int_{t_1}^{t_2} \partial_t \bW^{1}_{\eta}(s,\cdot) \dd s \Big\|_{L^1(\Omega)}\\
    &\overset{\eqref{eq:W_system}}{\leq}\bigg\|\int_{t_1}^{t_2}V_{1}(\bW^{1}_{\eta}(s,\cdot))\partial_{x}\bW^{1}_{\eta}(s,\cdot) \dd s\bigg\|_{L^1(\Omega)}\\
    &\quad+\bigg\|\int_{t_1}^{t_2}\!\!\!\!\tfrac{1}{\eta}\int_{*}^{\infty}\!\!\!\!\exp(\tfrac{*-y}{\eta})V_{1}'(\bW^{1}_{\eta}(s,y))\bW^{1}_{\eta}(s,y)\partial_{y}\bW^{1}_{\eta}(s,y)\dd y \dd s \bigg\|_{L^1(\Omega)}\\
    &\quad+\bigg\|\int_{t_1}^{t_2}\!\!\!\!\tfrac{1}{\eta}\int_{\ast}^{\infty}\!\!\!\!\!\exp(\tfrac{\ast-y}{\eta})\mathscr{S}\big(\bW_{\eta}(t,y),\eta\partial_{y}\bW_{\eta}(t,y),y\big)\dd y\bigg\|_{L^1(\Omega)}\\
    &\leq \|V_1\|_{L^\infty((0,\|\brho_{0}\|_{L^{\infty}(\R;\R^{2})}))} |\bW^{1}_{\eta}|_{L^\infty((0,T); TV(\R))} |t_1-t_2|\\
    &\quad+ \|V'_1\|_{L^\infty((0,\|\brho_{0}\|_{L^{\infty}(\R;\R^{2})}))} \|\bW^{1}_{\eta}\|_{L^\infty((0,T); L^\infty(\R))}|\bW^{1}_{\eta}|_{L^\infty((0,T); TV(\R))} |t_1-t_2|\\
    &\quad +\int_{t_{1}}^{t_{2}}\!\!\!\!\int_{\R}\Big|S\big(\bW^1_{\eta}(t,y)-\eta \partial_{2}\bW^1_{\eta}(t,y),\bW^2_{\eta}(t,y)-\eta \partial_{2}\bW^2_{\eta}(t,y),\bW^1_{\eta}(t,y),\bW^2_{\eta}(t,y), y\big)\Big|\tfrac{1}{\eta}\int_{-\infty}^{y}\!\!\!\!\!\!\e^{\frac{x-y}{\eta}}\dd x\dd y\dd s\\
      &\overset{\cref{ass:general}}{\leq} \left(\|V_1\|_{L^\infty((0,\|\brho_{0}\|_{L^{\infty}(\R;\R^{2})}))}+\|V'_1\|_{L^\infty((0,\|\brho_{0}\|_{L^{\infty}(\R;\R^{2})}))}\|\brho_0\|_{L^{\infty}(\R;\R^{2})}\right)
    |\brho_{0}|_{TV(\R;\R^{2})}|t_1-t_2|\\
    &\quad+2|t_{1}-t_{2}|\mathcal{H}_{BV}.
\end{align*}
After repeating the same computations for \(\bW^{2}_{\eta}\) and taking into account that the bounds obtained are uniform in the approximation of the initial data set, we establish the claim.
\end{proof}

Thanks to the compactness result given in \cref{thm:compactness}, we can ascertain (even directly) the convergence to a weak solution. Furthermore, due to the confirmation of entropy admissibility in \cref{thm:entropy_admissibility}, we also demonstrate the convergence to the entropy solution.
\begin{corollary}[Convergence to a weak (local) solution]\label{cor:convergence_weak}
For every sequence \(\{\eta_k\}_{k\in \N}\subset \R_{>0}\) with \(\lim_{k\to +\infty}\eta_k=0\) there exists a subsequence (denoted again by \((\eta_{k})_{k\in\N}\)) and a function \[\brho_\ast\in C\big([0,T];L^1_{\loc}(\R;\R^2)\big)\] so that the solution $\brho_{\eta_k}\in C\big([0,T]; L^1_{\loc}(\R;\R^2)\big)$ of the nonlocal system of balance laws, as given in~\cref{defi:nonlocal_system}, converges in \(C\big([0,T]; L^1_{\loc}(\R;\R^2)\big)\) to the limit function $\brho_{\ast},$. The same holds for the nonlocal term \(\bW_{\eta_k}\).
\end{corollary}
\begin{proof}
Applying \cref{thm:compactness} the set of nonlocal terms \(\bW_{\eta_k}\) is compact in \(C\big([0,T]; L^1_{\loc}(\R;\R^2)\big)\). This is why there exists a limit function \(\brho_{\ast}\in C\big([0,T];L^1_{\loc}(\R;\R^2)\big)\) such that
\begin{equation*}
    \lim_{k \to \infty} \|\bW_{\eta_k}-\brho_{\ast}\|_{C([0,T]; L^1_{\loc}(\R;\R^2))}=0.
\end{equation*} 
Thanks to \cref{eq:identity_nonlocal}, we can write, for \(t\in[0,T]\),
\begin{equation*}
    \|\bW_{\eta_k}(t,\cdot)-\brho_{\eta_k}(t,\cdot)\|_{L^1(\R;\R^2)}=\eta_k |\bW_{\eta_k}(t,\cdot)|_{TV(\R;\R^2)}\leq \eta_k |\brho_0|_{TV(\R;\R^2)}
\end{equation*}
and, thus, we also (as \(\lim_{k\rightarrow\infty}\eta_{k}=0\)) obtain 
\begin{equation*}
    \lim_{k \to \infty} \|\brho_{\eta_k}-\brho_{\ast}\|_{C([0,T]; L^1_{\loc}(\R^2))}=0.
\end{equation*} 
$\brho_\ast$ is a weak solution of the local system in \cref{defi:local_system} thanks to convergence in \(C\big([0,T]; L^1_{\loc}(\R;\R^2)\big),\) and due to the uniform bounds on \(\|\brho_{\eta_k}\|_{L^{\infty}((0,T);L^{\infty}(\R;\R^{2}))}\).
\end{proof}
This brings us to our final and most significant result. By bringing together the findings of the previous theorem, we ultimately assert the strong convergence of both the nonlocal term and nonlocal solution to the entropy solution of the local conservation law for \(\eta\rightarrow 0\).
\begin{theorem}[Convergence to the Entropy solution]\label{theo:singular_limit_problem}
Given \cref{ass:general}, the nonlocal term $\cW_\eta[\brho_\eta]$ and the corresponding nonlocal solution $\brho_\eta\in C\big([0,T];L^{1}_{\loc}(\R;\R^{2})\big)$ of the nonlocal system in \cref{defi:nonlocal_system} converge in $C\big([0,T];L^{1}_{\loc}(\R;\R^{2})\big)$ to the entropy solution of the corresponding local system of balance laws in \cref{defi:local_system}.
\end{theorem}
\begin{proof}
This is a direct consequence of \cref{cor:convergence_weak} and \cref{thm:entropy_admissibility}.
\end{proof}

\begin{remark}[Generalization to larger systems and more general kernels]\label{rem:generalizations}
~
\begin{description}
    \item[Larger Systems:]
By slightly adjusting the right-hand side of the system of nonlocal balance laws and imposing the corresponding assumptions on the source term, as reported in \cref{ass:general}, the same type of convergence can be proven for a system of any dimension (and not solely, as we did here, for \(N=2\)). The primary purpose of all our arguments is that the nonlocal fluxes decoupled. Coupling different equations within the fluxes might undermine the required uniform maximum principle. This will undoubtedly complicate any representation of the nonlocal terms in \cref{lem:nonlocal_transport_equation}.
\item[More general kernels:] It is very likely that the obtained convergence can be extended to more general kernels, such as a convex kernel, as described in \cite{Marconi2023}. The result should also hold for kernels with fixed support of the type reported in \cite{keimer42}. 
\end{description}
\end{remark}
\section{Numerical simulations}\label{sec:numerics}
In this section, we present several numerical simulations conducted using an Upwind-type numerical scheme, as detailed in~\cite{chiarello2019non-local, friedrich2018godunov}. In particular, we consider the source term 
\[S(\brho_1, \brho_2)\coloneqq\big(\brho_2-\brho_1\big)\chi_{[-2,2]}(x),\ x\in\R.\]  \Cref{fig:convergence_to_0} shows the convergence of the approximate nonlocal solution to the local one for decreasing values of \(\eta\). The corresponding $(t,x)-$plots are shown in \cref{fig:tx_plots}.
As can be observed, over time, the densities of both lanes converge due to the lane-changing behavior.

\begin{figure}
\centering
		\includegraphics[scale=0.35]{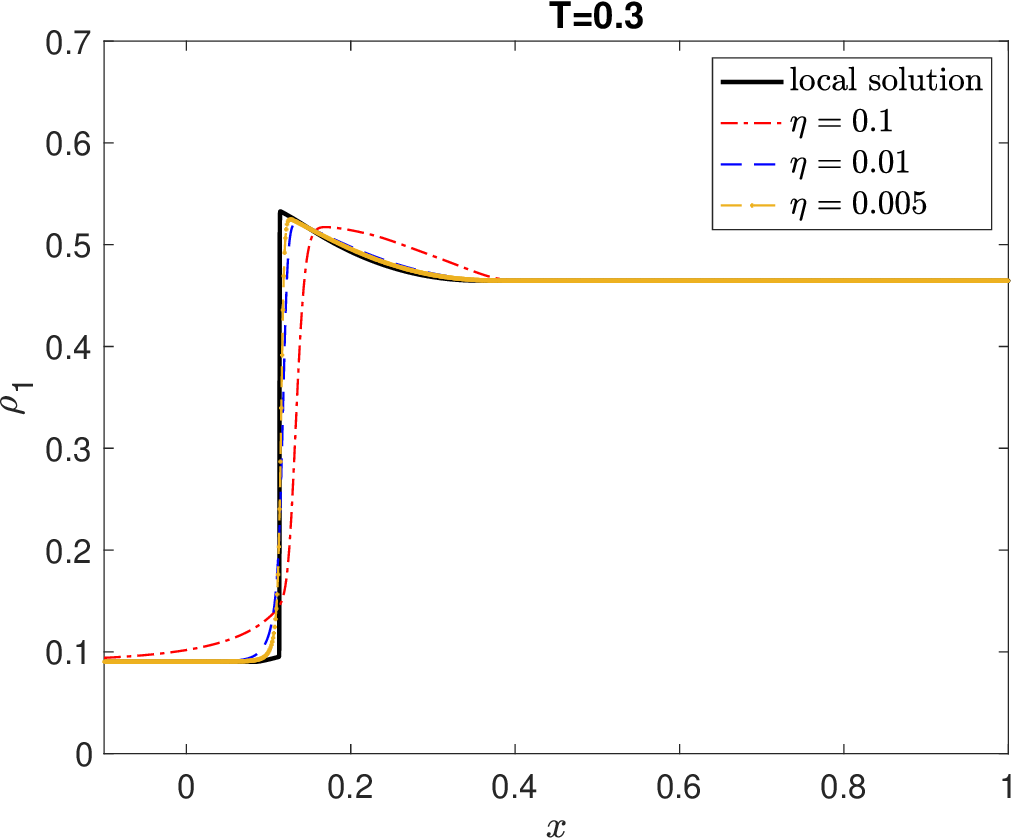}
		\includegraphics[scale=0.35]{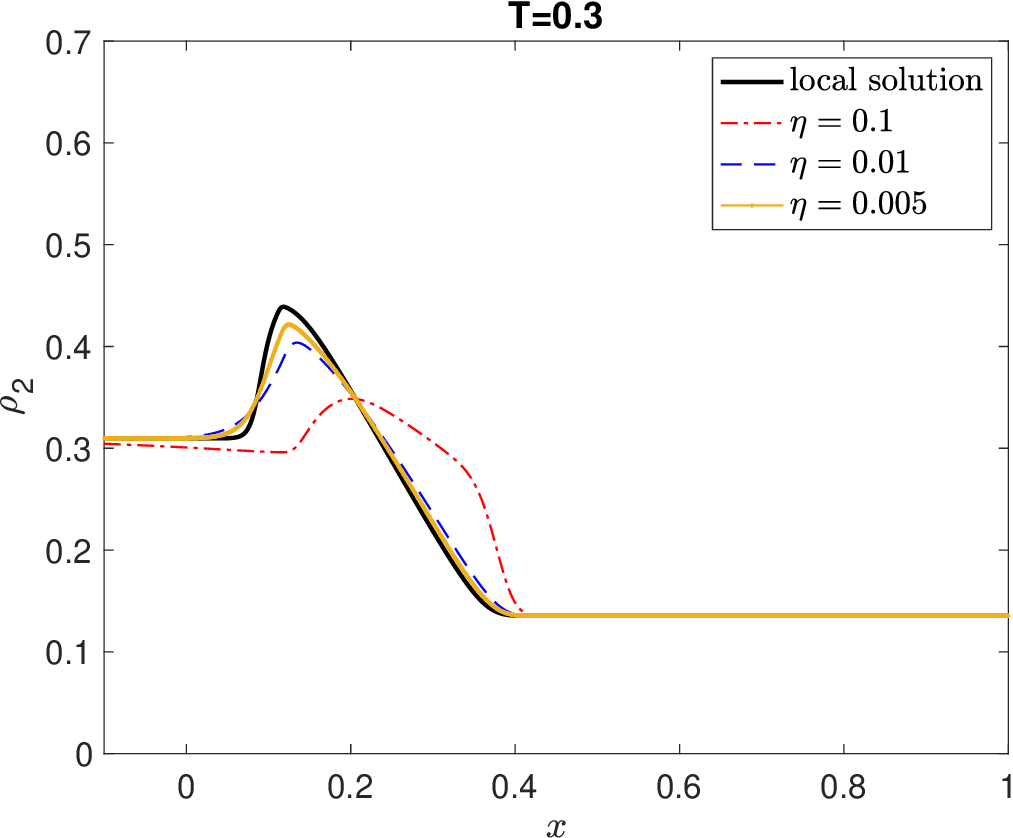}\\
		\includegraphics[scale=0.35]{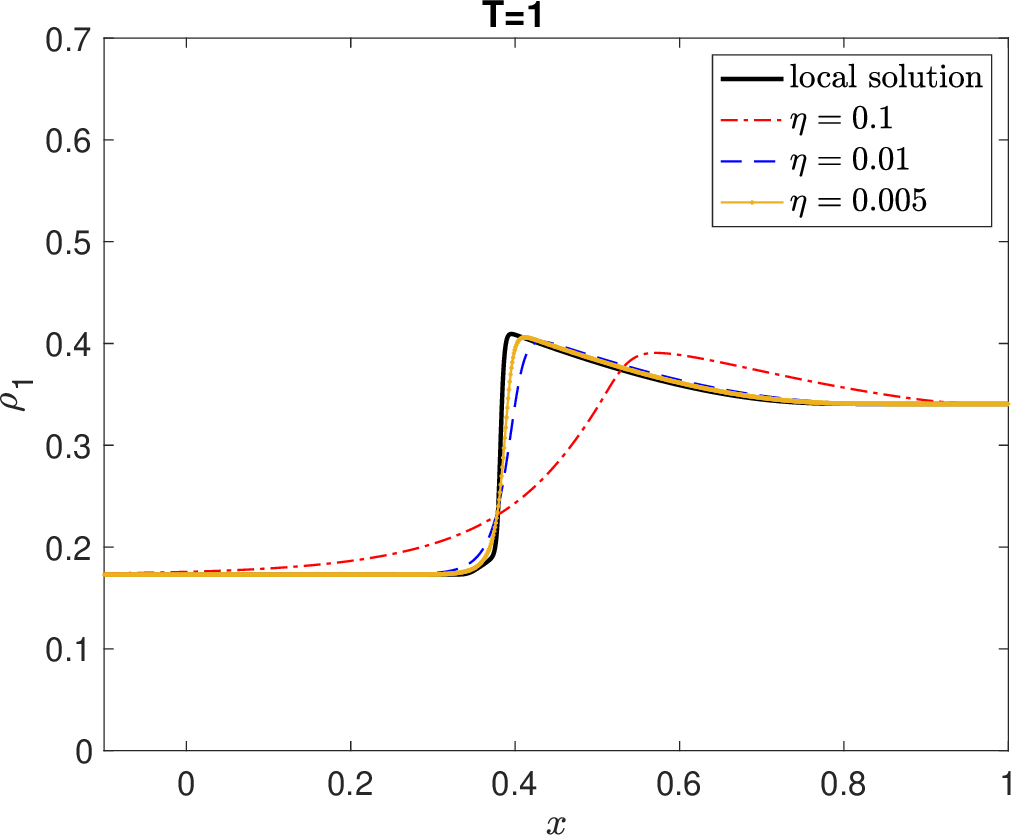}
		\includegraphics[scale=0.35]{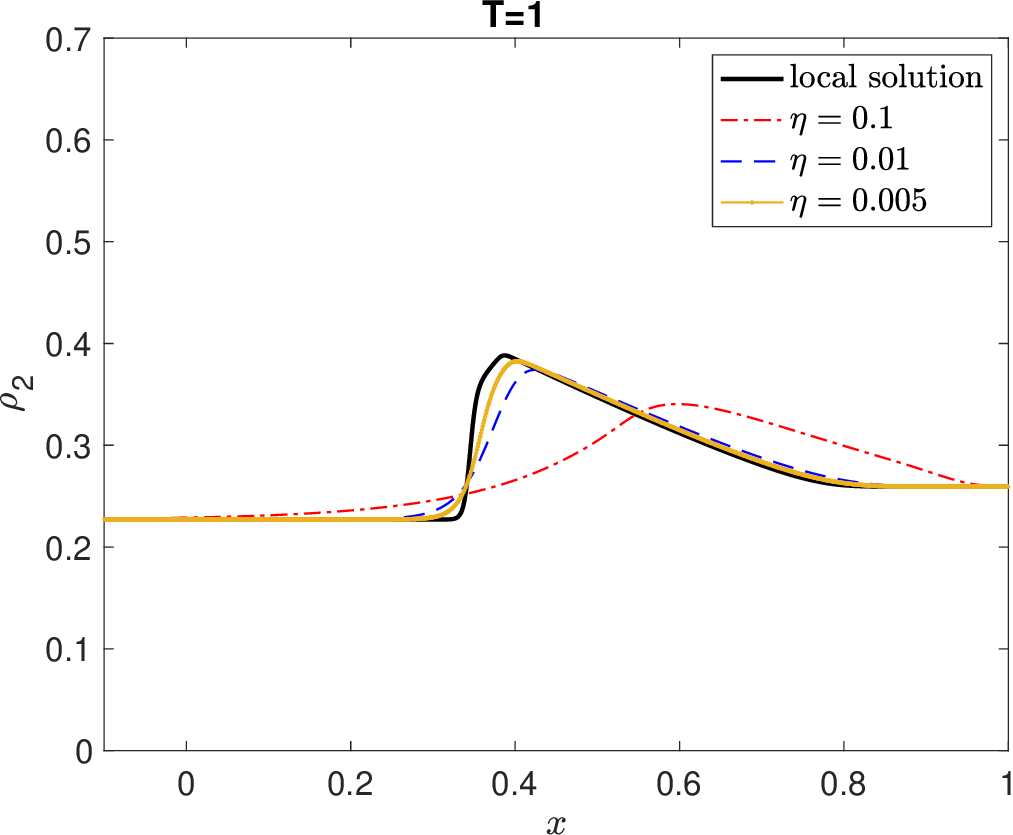}
	\caption{Convergence to the local solution for $\eta\to 0 $ with the initial datum \(\brho_{1}\equiv0.6\chi_{\R_{\geq0}}\) and \(\brho_{2}\equiv0.4\chi_{\R_{\leq0.1}}\) and time points \(T=0.3\) (top row) and \(T=1\) (bottom row),}
 \label{fig:convergence_to_0}
\end{figure}

\begin{figure}
\centering
	\hspace*{-1pt}
		\includegraphics[scale=0.27]{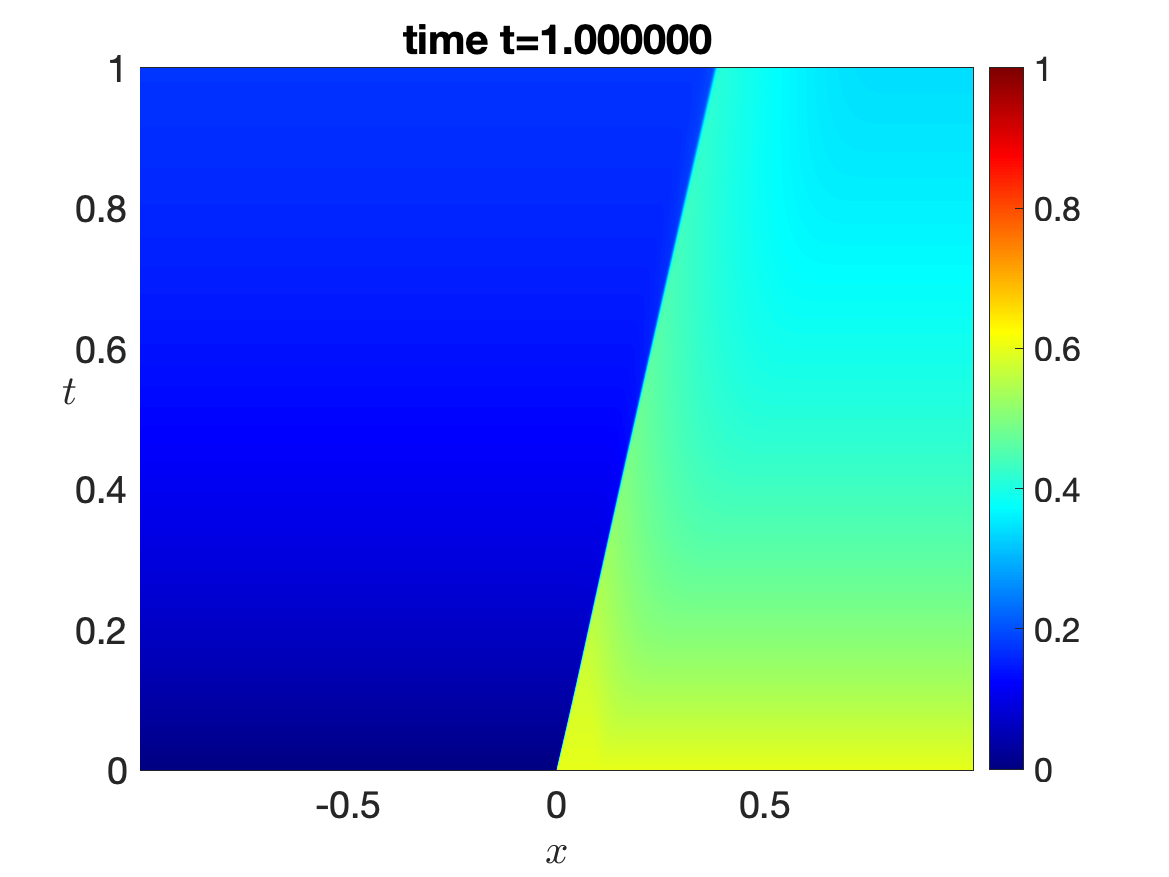}
  \includegraphics[scale=0.27]{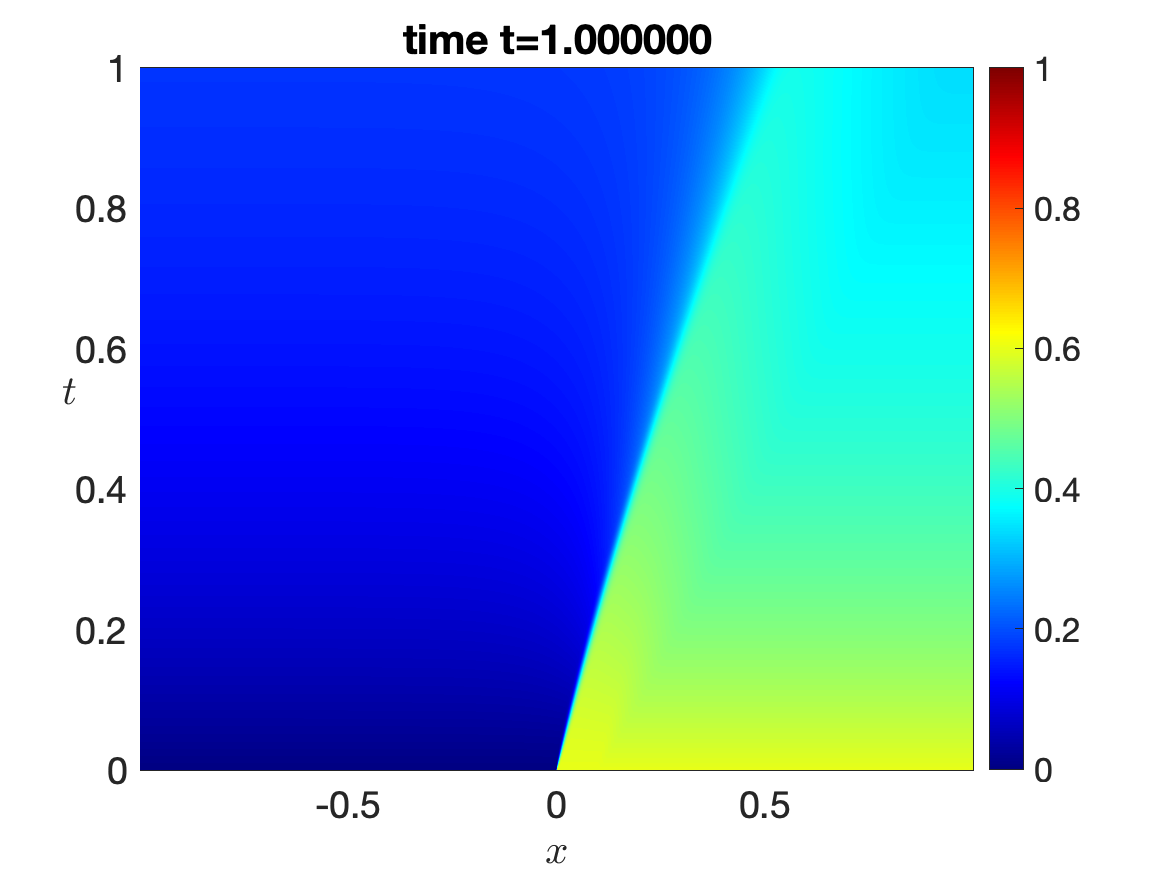}
	\includegraphics[scale=0.27]{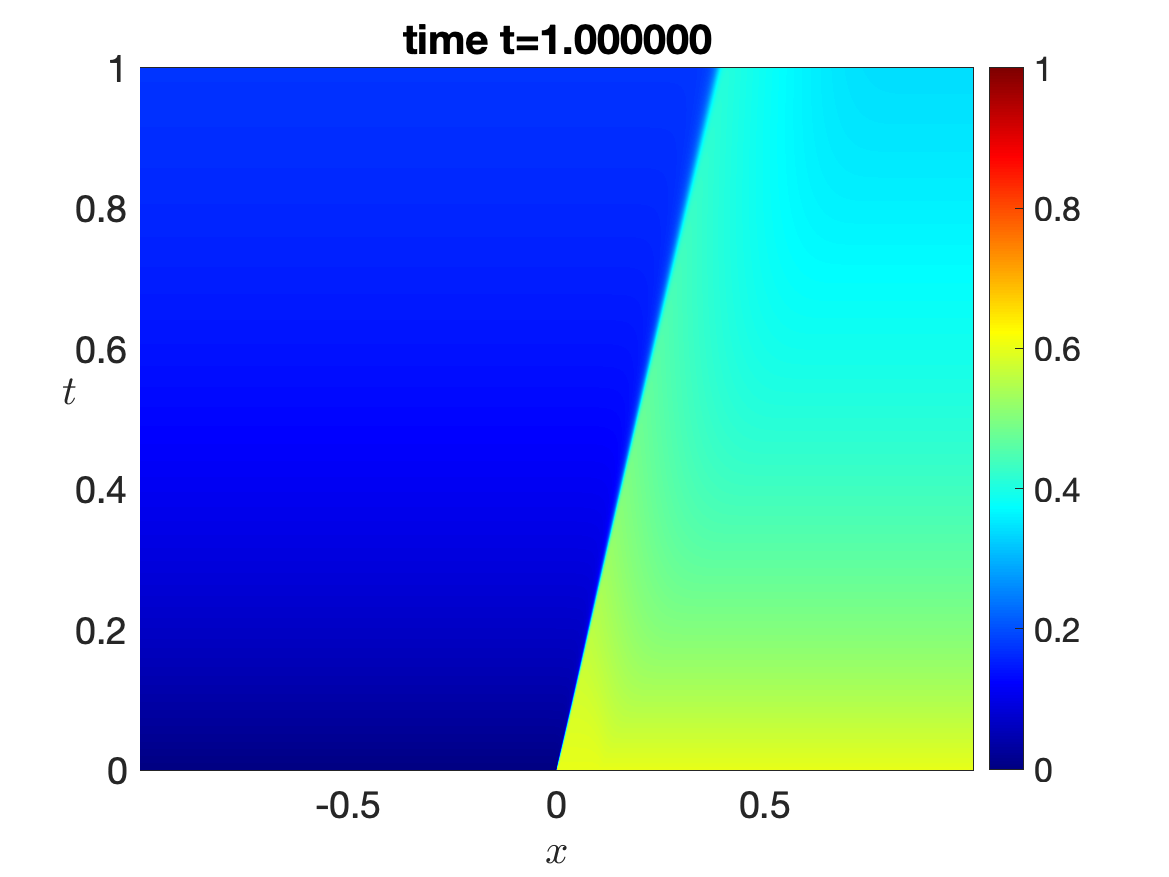}

   \includegraphics[scale=0.27]{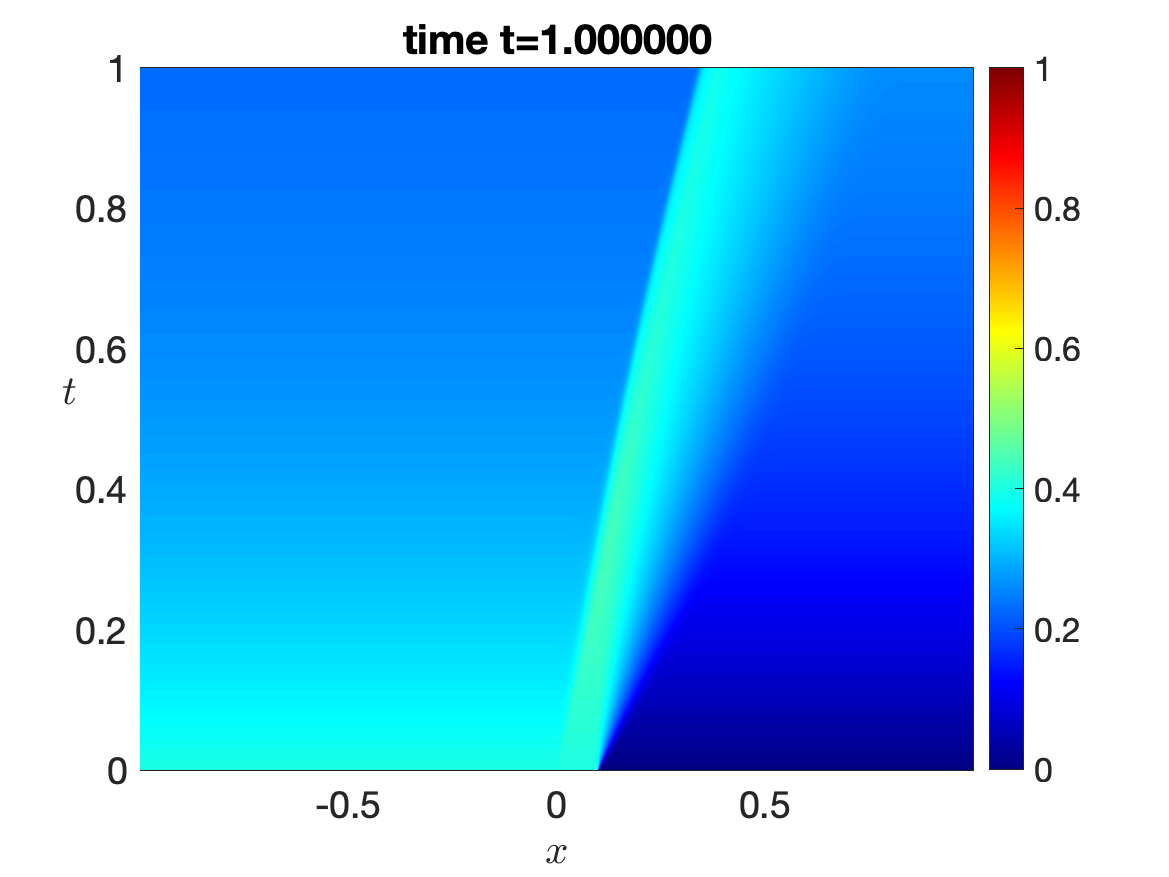}
		\includegraphics[scale=0.27]{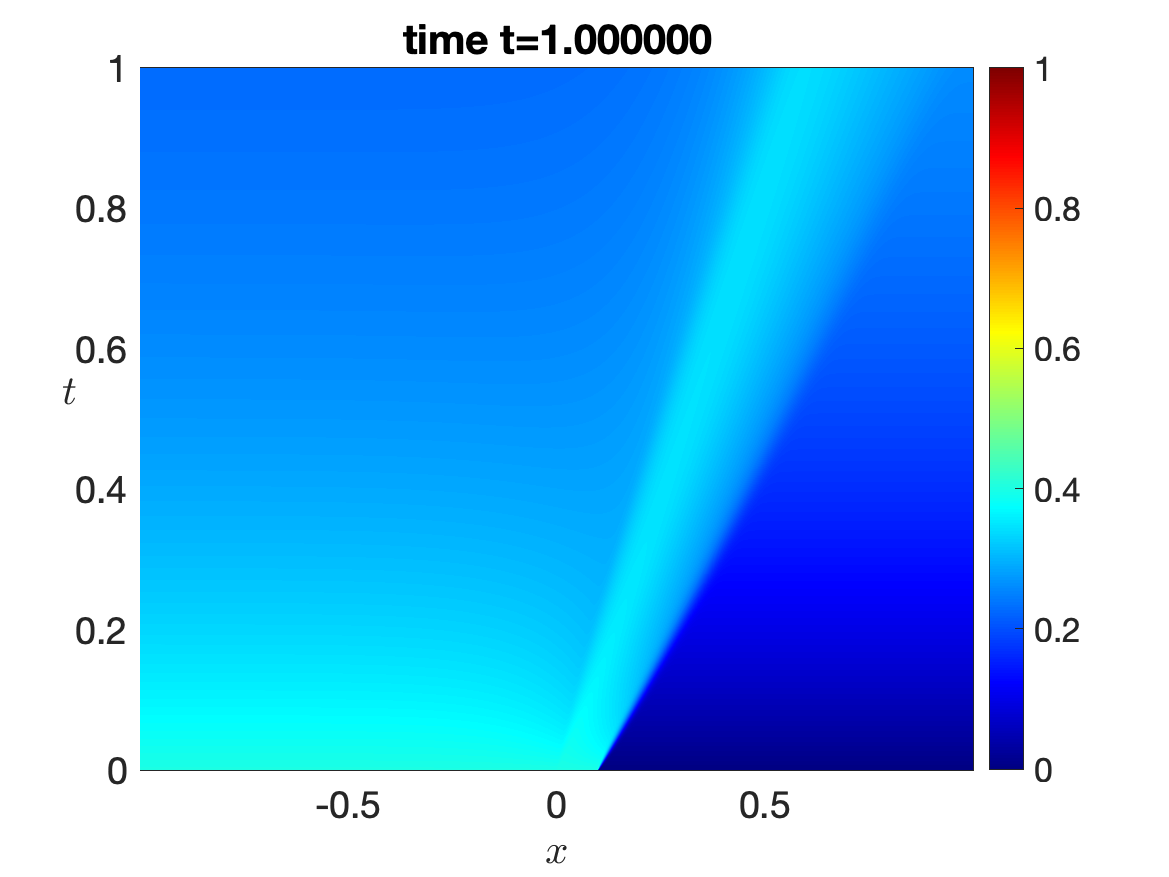}
  		\includegraphics[scale=0.27]{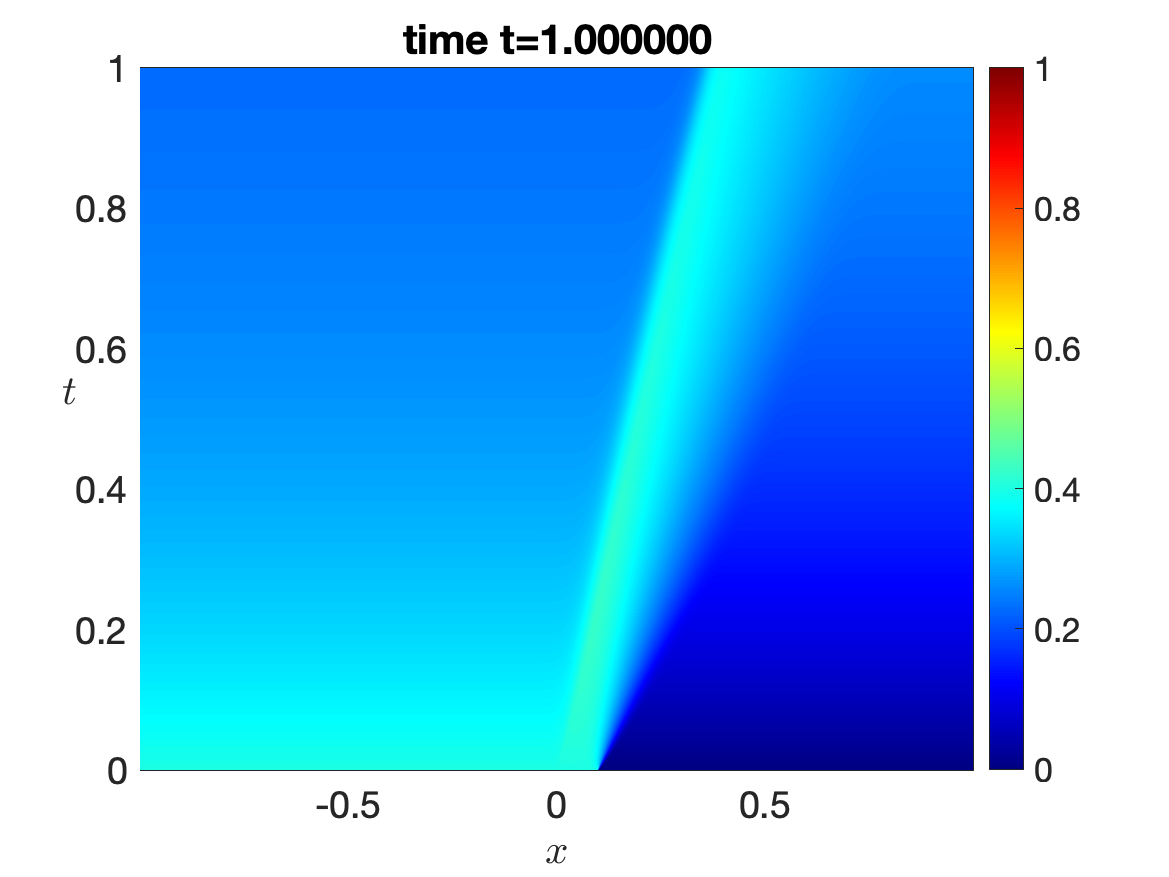}
	\caption{\((t,x)-\)plots of the local  and nonlocal solutions with \(\eta\in\{0, 0.1, 0.005\}\), from left to right. In the first row,  \(\brho_1\) is shown,  and in the second row, \(\brho_2\).} \label{fig:tx_plots}
\end{figure}
Clearly, the claimed convergence can be observed for smaller \(\eta\in\R_{>0}\). 
Moreover, in \Cref{fig:tv_plots}, the total variation is depicted as it varies with different values of $\eta$. Furthermore, it can be seen that, for the chosen source term \(S(\brho_1, \brho_2)\coloneqq\brho_2-\brho_1\) the total variation decreases (and not just finite as proven in  \cref{thm:tv_bound}). However, as anticipated in a nonlocal approximation, the total variation decreases as \(\eta\) increases.
\begin{figure}
\centering
	\hspace*{-1pt}
  \includegraphics[scale=0.40]{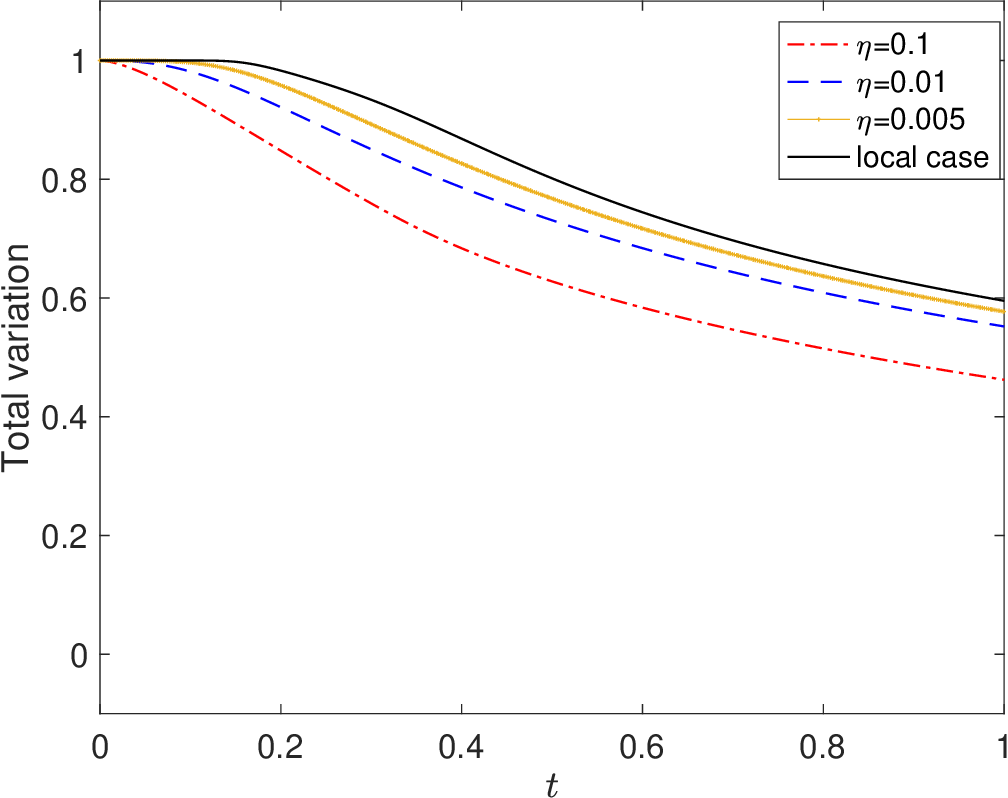}\\
\caption{Plots of the total variation \(\big|\brho^{1}_{\eta}(t,\cdot)\big|_{TV(\R)}+ \big|\brho^{2}_{\eta}(t,\cdot)\big|_{TV(\R)}\) for different $\eta\in\{0.1, 0.01, 0.005\}.$ } \label{fig:tv_plots}
\end{figure}

\section{Conclusions and open problems}\label{sec:conclusions}
In this paper, an analytical proof of nonlocal-to-local convergence for a system of balance laws, which models lane-changing traffic flow, was presented. Coupling occurred via the right-hand side. 
One crucial aspect was the ability to express the nonlocal system in terms of a system of nonlocal terms, facilitated by selecting the exponential kernel (though generalizations similar to \cref{rem:generalizations} should be readily achievable). 

The presented work, however, only scratches the surface of the singular limit problem for systems due to its ``weak'' coupling via the right-hand sides only. In a future study, it would be desirable to take into account coupling in the velocity functions of the dynamics.

Another interesting related problem involves investigating the singular limit problem for scalar nonlocal conservation laws in the context of bounded domains. Existence, uniqueness, and stability results have already been established i n this regard (for example, see \cite{KeimerPflugBounded2018,bayen2021boundary, colombo2018nonlocal}. However, in the system case, addressing the singular limit problem remains an open challenge. We currently lack the capability to obtain uniform \(TV\) estimates, and the manner in which we would converge to the boundary conditions, as defined by Bardos-Leroux-Nédélec \cite{BardosLerouxNedelec79} in the local case, remains unclear.

\bibliographystyle{plain}
\bibliography{biblio.bib}
\end{document}

%% file: definitions.tex
\newcommand{\e}{\mathrm{e}}

\newcommand{\R}{\mathbb R}
\newcommand{\N}{\mathbb N}

\newcommand{\brho}{\boldsymbol{\rho}}

\newcommand{\bW}{\boldsymbol{W}}

\newcommand{\eps}{\epsilon}

\renewcommand{\phi}{\varphi}

\newcommand{\mEF}{\mathcal{EF}}

\renewcommand{\:}{\mathrel{\coloneqq}}

\newcommand{\OT}{{\Omega_{T}}}

\newcommand{\loc}{\textnormal{loc}}

\newcommand{\dd}{\ensuremath{\,\mathrm{d}}}

\DeclareMathOperator{\sgn}{sgn}

\newcommand{\bq}{\ensuremath{\boldsymbol{q}}}

\newcommand{\cW}{\ensuremath{\mathcal{W}}}

\renewcommand{\epsilon}{\varepsilon}

\crefname{assumption}{Asm.}{Assumptions}
\crefname{definition}{Def.}{Definitions}